\documentclass[11pt]{amsart}
\usepackage{amsmath,amssymb,amsthm,mathtools}
\usepackage{mathrsfs}
\usepackage{xcolor}
\usepackage[colorlinks=true,linkcolor=blue,citecolor=blue,urlcolor=blue]{hyperref}
\usepackage[margin=1.2in]{geometry}
\newtheorem{theorem}{Theorem}
\newtheorem{lemma}[theorem]{Lemma}
\newtheorem{proposition}[theorem]{Proposition}
\newtheorem{corollary}[theorem]{Corollary}
\theoremstyle{definition}
\newtheorem{definition}[theorem]{Definition}
\newtheorem{assumption}{Assumption}
\newtheorem{remark}[theorem]{Remark}
\newtheorem{example}[theorem]{Example}
\newcommand{\R}{\mathbb{R}}

\newcommand{\rmd}{\mathrm{d}}
\newcommand{\La}{\mathscr{L}_a}
\newcommand{\Ls}{\mathscr{L}_s}
\newcommand{\Lo}{\mathscr{L}_{\mathrm{o}}}
\newcommand{\Lgen}{\mathscr{L}}
\newcommand{\Am}{\mathscr{A}_m}
\newcommand{\Pv}{\Pi_v}
\newcommand{\Lmu}[1][\mu]{L^2(#1)}
\newcommand{\Lmuz}[1][\mu]{L^2_0(#1)}
\newcommand{\ip}[3][\mu]{\langle #2,\, #3\rangle_{L^2(#1)}}
\newcommand{\nrm}[2][\mu]{\lVert #2\rVert_{L^2(#1)}}
\newcommand{\mux}{\mu_x}
\newcommand{\II}{\mathrm{II}}
\DeclareMathOperator{\Ran}{Ran}
\DeclareMathOperator{\Spec}{Spec}
\DeclareMathOperator{\diam}{diam}

\DeclareMathOperator{\supp}{supp}

\providecommand{\dist}{\operatorname{dist}}
\providecommand{\reach}{\operatorname{reach}}

\title[Hypocoercivity for specularly reflected Langevin dynamics]{Sharp hypocoercive convergence estimates for underdamped Langevin dynamics with specular reflection}
\author{Hengrong Du}
\author{Qi Feng}
\author{Lingjiong Zhu}
\date{\today}
\begin{document}
\begin{abstract}
We study the underdamped (kinetic) Langevin dynamics confined to a bounded domain $\Omega\subset\mathbb{R}^d$ by specular reflection of the velocity at the boundary. This process is the natural momentum-based analogue of the normally reflected overdamped Langevin diffusion, and it is used in practice for constrained sampling.  Assuming only that the position marginal $\mu_x\propto e^{-U}$ satisfies a Poincar\'e inequality on $\Omega$ with constant $m>0$, $\nabla^2U\succeq-K\,\mathrm{Id}$ and $\Omega$ is a convex domain, we prove that the law converges to the Gibbs measure exponentially fast in $L^2$, with an explicit rate that scales like $\sqrt m$, which is optimal when $U$ is convex. Since the normally reflected overdamped dynamics converges exactly at rate $m$, this establishes a square-root acceleration for constrained sampling in the small-gap regime when $m$ is small, matching the acceleration known in the unconstrained case. The explicit rate is the same in the unconstrained setting of Fan--Li--Lu. The proof adapts the modified $L^2$ hypocoercivity method of Dolbeault--Mouhot--Schmeiser with the gap-shifted corrector of Fan--Li--Lu. The specular symmetry makes the transport operator antisymmetric, and that the corrector automatically selects the Neumann realization of the overdamped generator, which is precisely the boundary condition that keeps every auxiliary function inside the specular class.  The Bochner identity used in the whole-space argument is replaced by a weighted Reilly formula, whose boundary contribution involves the second fundamental form of $\partial\Omega$ and is nonnegative for convex $\Omega$. Finally, we extend our results to the setting where the domain $\Omega$ is non-convex. We obtain an explicit contraction rate that depends on the domain.
\end{abstract}
\maketitle
\section{Introduction}\label{sec:intro}
Consider the problem of sampling a distribution $\mu_{x}$
on a constrained domain $\Omega\subset\mathbb{R}^{d}$
with probability density function
\begin{equation}
\mu_{x}(x)\propto\exp(-U(x)),\qquad x\in \Omega,
    \label{eq-target}
\end{equation}
for a function $U:\overline\Omega \to \mathbb{R}$.
Both constrained sampling (with domain $\Omega\subset\mathbb{R}^{d}$) and unconstrained sampling (with domain $\Omega=\mathbb{R}^{d}$) are fundamental in many machine learning applications, such as Bayesian statistical inference, Bayesian formulations of inverse problems, and Bayesian classification and regression tasks \cite{gelman1995bayesian,stuart2010inverse,andrieu2003introduction,teh2016consistency,DistMCMC19,GHZ2022,GIWZ2024}.

To the best of our knowledge, the first Langevin-based algorithm
for the constrained sampling problem \eqref{eq-target}
was the \textit{projected Langevin Monte Carlo} (PLMC) algorithm proposed and studied by \cite{bubeck2015finite, bubeck2018sampling}. PLMC can be viewed as the discretization of the \textit{reflected Langevin dynamics} (RLD),
i.e.\ the continuous-time overdamped Langevin stochastic differential equation (SDE) with (normal) reflection at the boundary:
\begin{equation}\label{reflected:SDE}
dX_{t}=-\nabla U(X_{t})\,\rmd t+\sqrt{2}\,\rmd W_{t}-n(X_{t})\ell(\rmd t),
\end{equation}
where $W_{t}$ is a standard $d$-dimensional Brownian motion, $n(x)$ is the \emph{outward} unit normal vector for $x\in\partial\Omega$ (and can be chosen arbitrarily off $\partial\Omega$),
and $-\int_{0}^{t}n(X_s)\,\ell(ds)$ is a bounded-variation reflection term that enforces the constraint
$X_t\in\overline\Omega$ for all $t\ge0$ whenever $X_0\in\overline\Omega$. The measure $\ell(\rmd t)$ is nonnegative with $\ell([0,t])<\infty$ for each $t$,
and it charges only boundary times, i.e.
$\supp(\ell)\subseteq \{t\ge0:\,X_t\in\partial\Omega\}$.
The reflected Langevin dynamics \eqref{reflected:SDE} can be viewed as a Skorokhod problem on $\Omega$ \cite{LionsSznitman84,Saisho87}:
one first follows the unconstrained Langevin increment, and whenever the trajectory hits the boundary,
a bounded-variation correction is added so that $X_t \in \overline\Omega$ for all $t\ge 0$.
Geometrically, this correction acts only at boundary times (i.e., the measure $\ell(\rmd t)$ is supported on
$\{t: X_t\in \partial \Omega\}$) and pushes the trajectory inward along $-n$.
In smooth domains, the reflection term is the \emph{minimal} correction that prevents leaving $\Omega$; existence and uniqueness for \eqref{reflected:SDE} on convex or smooth domains are classical \cite{LionsSznitman84,Saisho87}.

There have been many variants of Langevin algorithms for constrained sampling since the seminal works of \cite{bubeck2015finite, bubeck2018sampling}.
\cite{Lamperski2021} considers the \textit{projected stochastic gradient Langevin dynamics} in the setting of non-convex smooth Lipschitz $U$ on a convex body where the gradient noise is assumed to have finite variance with a uniform sub-Gaussian structure; see also \cite{zheng2022constrained}.
\cite{Brosse} proposed the \textit{proximal Langevin Monte Carlo} for constrained sampling; see also \cite{SR2020} for a study on the proximal stochastic gradient Langevin algorithm
from a primal-dual perspective and \cite{2026IslamZhu} for a study in the decentralized setting.
\textit{Mirror descent-based Langevin algorithms} (see e.g. \cite{hsieh2018mirrored,Chewi2020,Zhang2020,TaoMirror2021,Ahn2021})
are another alternative for constrained sampling, first proposed in \cite{hsieh2018mirrored}, inspired by the classical mirror descent in optimization.
More recently, inspired by the penalty method in the optimization literature,
\textit{penalized Langevin Monte Carlo algorithms} were proposed and studied in \cite{GHZ2022},
where the objective $U$ can be non-convex in general. In addition, constrained non-convex exploration combined with replica-exchange Langevin dynamics was proposed and studied in \cite{constraint_replica}.

In the literature, it is known
that in the setting of unconstrained sampling,
by \textit{breaking reversibility},
a \textit{non-reversible} Langevin dynamics
can converge to the target distribution faster on the Euclidean space $\mathbb{R}^{d}$.
This phenomenon is well understood through non-asymptotic
convergence analysis; see e.g. \cite{HHS93,HHS05,reyLDP,HWGGZ20,GGZ2}.
Motivated by this, in the setting of constrained sampling, non-reversible Langevin algorithms with skew-reflected boundary were
proposed and studied in \cite{DFTWZ2025}.
\cite{DFTWZ2025} showed that the discretized algorithm based on \textit{skew-reflected non-reversible Langevin dynamics} (SRNLD),
obtained by adding a \textit{skew-symmetric} matrix $J$
to the drift term in \eqref{reflected:SDE} and replacing
the normal reflection by a \textit{skew reflection},
can accelerate the convergence
of the (reversible) projected Langevin Monte Carlo.
However, how to choose the skew-symmetric matrix $J$ in practice is a challenging problem.
Very recently,
\cite{wang2026constrained} showed acceleration for constrained sampling through the lens of large deviations and asymptotic variance for SRNLD under the
additional constraint that $J(x)n(x)=0$. Indeed, under this additional condition,
skew-reflection becomes a normal reflection. Extensive numerical experiments using both synthetic and real data in \cite{wang2026constrained} demonstrated superior performance for the choice of $J$ that satisfies this additional condition.

One of the most popular non-reversible Langevin dynamics in the literature
is the \textit{underdamped Langevin dynamics} (ULD), which is a kinetic, momentum-based variant of the Langevin SDE
that can converge to the stationary distribution
faster than the overdamped Langevin dynamics \cite{Eberle,Ma2019,GGZ,CLW23,FLL26}.
Constrained versions of the underdamped dynamics have received considerably less attention than their constrained overdamped or unconstrained underdamped counterparts, and the available guarantees are of a different nature from ours. The most extensively studied model is the penalized underdamped Langevin \cite{GHZ2022} that replaces the hard constraint by a smooth penalty and analyses both penalized overdamped and penalized underdamped Langevin Monte Carlo algorithms, obtaining non-asymptotic iteration complexity bounds for possibly non-convex $U$ whose dimension dependence improves from $d$ to $\sqrt d$ upon passing to the underdamped scheme. For log-concave targets on a convex body, \cite{RandMidpointConstrained} analyses randomized-midpoint discretizations of kinetic dynamics under constraints. On the modeling side, \cite{NordenhogSharma25} builds a constrained score-based generative model on exactly the process \eqref{eq:sde} studied here --- underdamped Langevin with specular reflection of the velocity --- and compares it numerically with reflected (local-time) diffusions; but the analysis there concerns the discretization error rather than the mixing rate. On the kinetic partial differential equation (PDE) side, well-posedness of \eqref{eq:sde} and the associated trace theory are treated in \cite{BossyJabir15,Carrillo98,Mischler10}, constructive $L^2$ hypocoercivity in bounded domains under general Maxwell boundary conditions --- including the vanishing-accommodation (specular) endpoint --- is established in \cite{BCMT23} for the linearized Boltzmann and Landau collision operators in the potential-free setting, with mass-only-conserving operators such as relaxation and Fokker--Planck indicated there as straightforward adaptations, and the diffusion limit of the specularly reflected kinetic Fokker--Planck equation towards the normally reflected diffusion is derived in \cite{CesbronHutridurga16}.

It is also well known that when $U$ is $m$-strongly convex,
overdamped Langevin dynamics converges to the stationary distribution
exponentially fast in time $t$ with the speed $O(m)$; this is the Bakry--\'Emery criterion, see e.g.\ \cite[Ch.~4]{BGL14}.
Moreover, the same speed $O(m)$ can be achieved with the normal reflection \eqref{reflected:SDE}; see e.g. \cite{eberle-2016}.
However, even though several works have studied
the constrained underdamped Langevin dynamics with specular boundary,
an explicit exponential decay rate for \eqref{eq:sde} itself is missing from the literature,
and whether it enjoys the same acceleration as in the unconstrained case \cite{Eberle,CLW23,FLL26}
compared to the overdamped counterpart \eqref{reflected:SDE} has been an open problem until
a very recent work \cite{EberleLorlerJFA26}, 
that was the first to establish the $O(\sqrt m)$ rate for
underdamped Langevin dynamics with specular reflection --- on a Riemannian manifold
with locally convex boundary, and for randomized Hamiltonian Monte Carlo as well ---
by way of a quantitative space-time divergence lemma and the theory of second-order
lifts. Our paper is complementary to \cite{EberleLorlerJFA26}: we obtain the same scaling by an
entirely different route, adapting the modified $L^2$ method of
Dolbeault--Mouhot--Schmeiser \cite{DMS09,DMS15} with the gap-shifted corrector of
\cite{FLL26} directly to the boundary-value problem \eqref{eq:bke}, and this yields
three distinctions from \cite{EberleLorlerJFA26}. First, the universal constants improve
substantially, i.e. our Theorem~\ref{thm:main} gives the contraction rate $\Lambda=\frac{2-\sqrt2}{12}\sqrt m$
when $U$ is convex, hence a relaxation time at most $31.8/\sqrt m$, against the bound
$(1773+\pi)/\sqrt m$ of \cite[Cor.~17]{EberleLorlerJFA26}. Second, we prove that the $\sqrt m$
scaling cannot be improved for a bounded domain (Proposition~\ref{prop:sharp}), by computing the exact
eigenfunctions of the specularly reflected generator on an interval in a one-dimensional example. Third, in the Euclidean space, 
\cite{EberleLorlerJFA26} requires the convexity of $\Omega$, 
whereas we are able to work with non-convex domains (Section~\ref{sec:nonconvex})
by leveraging the fact that the geometry of $\Omega$ enters our argument
through a single inequality, the sign of the boundary term in a weighted Reilly
formula, so that replacing that sign condition by a lower bound on the second
fundamental form yields an explicit rate on domains that are not convex.

Quantitative hypocoercivity for kinetic equations in bounded domains is by now an active subject; in particular Bernou--Carrapatoso--Mischler--Tristani \cite{BCMT23} established constructive $L^2$ hypocoercivity for linear kinetic equations (linearized Boltzmann and Landau; mass-only-conserving models such as relaxation and Fokker--Planck by adaptation) under general Maxwell boundary conditions, including specular reflection, and boundary trace theory for kinetic Fokker--Planck goes back at least to \cite{Carrillo98}. For the underdamped Langevin dynamics itself, explicit whole-space rates with the optimal $O(\sqrt m)$ scaling were obtained in \cite{CLW23} via a space-time Poincar\'e inequality and revisited via the modified $L^2$ method in \cite{FLL26}; see also \cite{RS18,Herau06,Villani09} for background on the $L^2$ and hypocoercivity frameworks. To the best of our knowledge, the present paper is the first to record that the gap-shifted Dolbeault--Mouhot--Schmeiser corrector of \cite{FLL26} is fully compatible with specular boundary conditions, with no degradation of the rate.

Let $\Omega\subset\R^d$ be a bounded convex domain with $C^{2,1}$ boundary, with outward unit normal $n(x)$ at $x\in\partial\Omega$, and let $U\in C^2(\overline\Omega)$. We study the underdamped Langevin dynamics $(X_t,V_t)\in\overline{\Omega}\times\R^d$ confined to $\Omega$ by \emph{specular reflection}:
\begin{subequations}\label{eq:sde}
\begin{align}
\rmd X_t&=V_t\,\rmd t,\\
\rmd V_t&=-\nabla U(X_t)\,\rmd t-\gamma V_t\,\rmd t+\sqrt{2\gamma}\,\rmd W_t,\qquad X_t\in\Omega,\\
V_{t+}&=R_{X_t}V_{t-},\qquad X_t\in\partial\Omega,\label{eq:sde-reflect}
\end{align}
\end{subequations}
where $\gamma>0$ is the friction coefficient, $W_t$ a standard $d$-dimensional Brownian motion, and
\begin{equation}\label{eq:reflection-map}
R_xv:=v-2\,(v\cdot n(x))\,n(x)
\end{equation}
is the specular reflection map, which flips the normal component of the velocity and preserves the tangential component and the kinetic energy. Well-posedness of \eqref{eq:sde} has been studied in \cite{BossyJabir15,BossyJabirMaftei17,Carrillo98}; the same process is the object of the numerical schemes of \cite{LST24}. See Sections~\ref{sec:sde-pde}--\ref{sec:wp} for the precise relationship between \eqref{eq:sde} and the boundary-value problem \eqref{eq:bke}, for what is available in the literature, and for the functional-analytic input we assume.

We record for later use the three elementary properties of $R_x$ that the whole paper rests on: for every $x\in\partial\Omega$,
\begin{equation}\label{eq:R-properties}
R_x^2=\mathrm{Id},\qquad |R_xv|=|v|,\qquad (R_xv)\cdot n(x)=-\,v\cdot n(x),
\end{equation}
and consequently $R_x$ is a linear isometry of $\R^d$ with $\det R_x=-1$, so $|\det R_x|=1$.
The law $\varrho(t,x,v)$ of \eqref{eq:sde} solves the kinetic Fokker--Planck equation on $\Omega\times\R^d$,
\begin{equation}\label{eq:kfp}
\partial_t\varrho=\left(-v\cdot\nabla_x+\nabla_xU\cdot\nabla_v\right)\varrho
+\gamma\left(\nabla_v\cdot(v\varrho)+\Delta_v\varrho\right),
\end{equation}
with the specular boundary condition, expressed in terms of boundary traces as
\begin{equation}\label{eq:specularBC}
\varrho(t,x,v)=\varrho(t,x,R_xv),\qquad x\in\partial\Omega,\ v\in\R^d.
\end{equation}
With the Hamiltonian $H(x,v):=\frac12|v|^2+U(x)$, the Gibbs density
\begin{equation}\label{eq:gibbs}
\mu(x,v):=\frac1Z\,e^{-H(x,v)},\qquad Z:=\int_{\Omega\times\R^d}e^{-H(x,v)}\,\rmd x\,\rmd v,
\end{equation}
is a stationary solution of \eqref{eq:kfp}--\eqref{eq:specularBC}: it satisfies \eqref{eq:specularBC} because $|R_xv|=|v|$. As in the whole-space case, $\mu$ is a product measure, $\rmd\mu=\rmd\mux(x)\,\rmd\kappa(v)$, where $\mux\propto e^{-U(x)}\mathbf{1}_\Omega\,\rmd x$ and $\kappa\propto e^{-\frac{1}{2}|v|^{2}}\,\rmd v$ is the standard Gaussian distribution on $\R^d$.
The ergodic behavior of \eqref{eq:sde} is encoded in the backward Kolmogorov equation:
\begin{equation}\label{eq:bke}
\partial_tf=\Lgen f:=\left(v\cdot\nabla_x-\nabla_xU\cdot\nabla_v\right)f
+\gamma\left(-v\cdot\nabla_v+\Delta_v\right)f\quad\text{in }\Omega\times\R^d,
\end{equation}
with the specular boundary condition:
\[
f(t,x,v)=f(t,x,R_xv),\qquad\text{for $x\in\partial\Omega$}.
\]

In a recent seminal paper \cite{FLL26}, Fan, Li and Lu showed that in the whole space $\R^d\times\R^d$ the modified $L^2$ method of Dolbeault--Mouhot--Schmeiser \cite{DMS09,DMS15}, with the standard corrector replaced by the \emph{gap-shifted corrector}:
\begin{equation}\label{eq:gap-corrector}
\Am=(m-\Lo)^{-1}(\La\Pv)^*,
\end{equation}
yields the optimal $O(\sqrt m)$ hypocoercive rate previously obtained via space-time Poincar\'e inequalities \cite{CLW23} --- themselves building on the variational approach of \cite{AAMN24} --- and via lifting techniques \cite{BrigatiLorlerWang25,EGHLM25,EberleLorler26,LiLu25}; here $m$ is the Poincar\'e constant of $\mux$, $\Lo$ the overdamped infinitesimal generator, $\La$ the Liouville operator and $\Pv$ the velocity average (all defined in Section~\ref{sec:setting}).
In our paper, we will show that this method extends, with essentially no loss and with identical constants in the contraction rate, to the boundary-value problem \eqref{eq:bke} with specular reflection on a bounded convex domain.
We will also extend our analysis to a non-convex domain by obtaining the contraction rate that depends on the domain, where a convexity assumption on the boundary is replaced by a lower bound on the second fundamental form.
We control the unfavorable
boundary contribution by interior quantities. This is closely related to
techniques used in \cite{Wang14modified,CTT18} for functional inequalities and reflecting diffusions
on manifolds with non-convex boundary. However, the conformal-change approach used in \cite{Wang14modified,CTT18} is not directly transferable to the present kinetic
problem without also changing the metric notion of specular reflection.
The point used here is instead that the boundary trace absorption is
compatible with the gap-shifted corrector and the hypocoercivity proof continues to hold, albeit with a modified constant.

\noindent\textbf{Contributions.} Our contributions can be summarized as follows.
\begin{itemize}
\item
First, we provide an explicit exponential convergence rate for underdamped Langevin dynamics with specular reflection on a bounded convex domain (Theorem~\ref{thm:main}), and we show that it exhibits the $\sqrt m$ scaling, which is optimal when $U$ is convex (Proposition~\ref{prop:sharp}), hence a square-root acceleration over the normally reflected overdamped dynamics. As a consequence, we obtain quantitative convergence rates in $\chi^{2}$-divergence,
total variation and Wasserstein distances (Corollary~\ref{cor:divergences}).
\item
Second, we identify the two structural mechanisms that make the modified $L^2$ method survive the presence of a boundary --- specular symmetry, which restores the antisymmetry of the transport operator, and the automatic selection of the \emph{Neumann} realization of $\Lo$ by the corrector, which is exactly the condition keeping every auxiliary function inside the specular class. Neither is a formality: if one insists on the Dirichlet realization of $\Lo$, the function $v\cdot\nabla_x h$ generated by the corrector is \emph{not} specular and every integration by parts in Lemma~\ref{lem:corrector} produces an uncontrolled boundary flux; and if one works with a test-function class not adapted to \eqref{eq:reflection-map}, then $\La$ is not antisymmetric and already the energy identity \eqref{eq:energy} fails.
\item
Third, we replace the Bochner identity, the only ingredient of \cite{FLL26} that sees the geometry of the spatial domain, by a weighted Reilly formula (Lemma~\ref{lem:reilly}) whose boundary term carries the second fundamental form of $\partial\Omega$; convexity makes this term nonnegative, so the flat-space second-order estimate survives with the \emph{same} constant. The net effect is that every constant in \cite{FLL26} carries over unchanged, with $m$ now the Neumann Poincar\'e constant of $\mux$ on $\Omega$.
\item
Finally, we extend our results to the setting where the domain
$\Omega$ is non-convex. We obtain an explicit contraction rate that depends
on the domain (Theorem~\ref{thm:main-nc}).
The analysis is based on the observation that
convexity enters the whole argument at a
single place, the sign of the boundary term
$\int_{\partial\Omega}\II(\nabla h,\nabla h)\,e^{-U}\rmd\sigma$ in the
weighted Reilly formula, and that this term is only ever needed for
Neumann functions $h$, whose gradients are tangential. If the boundary is
merely semiconvex, $\II\succeq-\sigma_0$, the unfavorable part of the term
is bounded by $\sigma_0\int_{\partial\Omega}|\nabla h|^2e^{-U}\rmd\sigma$,
and a trace inequality built from a Lipschitz extension of the normal
field (Lemma~\ref{lem:trace-nc}) controls it by
$\theta\|\nabla^2h\|^2+C_\theta\|\nabla h\|^2$ with $\theta<1$ as small as
we wish. Since the corrector bound of Lemma~\ref{lem:corrector} tolerates
a constant factor in front of the Hessian term, the second-order estimate
survives with $K$ replaced by an explicit
$K_\theta=K+\sigma_0^2/\theta+\sigma_0\Theta$, where $\Theta$ is
controlled by the reach of $\partial\Omega$ and $\|\nabla U\|_\infty$
(Remark~\ref{rem:Theta-size-nc}). The negative boundary curvature thus enters the rate in the
same way as the negative curvature $K$ of the potential.
\end{itemize}

\noindent\textbf{Organization.} The rest of the paper is organized as follows.
We introduce the technical background and preliminary results in Section~\ref{sec:setting}.
The main results of the paper are presented in Section~\ref{sec:main}.
The key technical lemmas that are used to prove the main result are provided in Section~\ref{sec:lemmas}.
Section~\ref{sec:proof} presents the proof of the main result.
Next, we extend our results to the non-convex domain in Section~\ref{sec:nonconvex}.
Finally, we conclude in Section~\ref{sec:conclusion}.

\section{Preliminaries}\label{sec:setting}

\subsection{Specular class and basic operators}
Throughout, $\Omega\subset\R^d$ is a bounded convex domain with $\partial\Omega\in C^{2,1}$ (or $C^3$) and $U\in C^2(\overline\Omega)$. We write $\rmd\mu=\rmd\mux\,\rmd\kappa$ as in Section~\ref{sec:intro}, $\ip{f}{g}=\int_{\Omega\times\mathbb{R}^{d}} fg\,\rmd\mu$ for the (real) $L^2(\mu)$ inner product, and $\Lmuz=\{f\in\Lmu:\int_{\Omega\times\mathbb{R}^{d}} f\rmd\mu=0\}$.
\begin{definition}[$C^{k,1}$ boundary]\label{def:Ck1}
Let $k\ge0$ be an integer. A bounded open set $\Omega\subset\R^d$ has
\emph{boundary of class $C^{k,1}$}, written $\partial\Omega\in C^{k,1}$, if for
every $x_0\in\partial\Omega$ there exist $r,h>0$, an orthonormal coordinate
system $y=(y',y_d)\in\R^{d-1}\times\R$ centered at $x_0$, and a function
$\varphi\in C^{k,1}(\overline{B'_r})$ with $\varphi(0)=0$ and
$\sup_{B'_r}|\varphi|<h$, such that, writing
$B'_r:=\{y'\in\R^{d-1}:|y'|<r\}$ and $V:=B'_r\times(-h,h)$,
\begin{equation}\label{eq:local-graph}
\Omega\cap V=\big\{(y',y_d)\in V:\ y_d>\varphi(y')\big\},
\qquad
\partial\Omega\cap V=\big\{(y',y_d)\in V:\ y_d=\varphi(y')\big\}.
\end{equation}
Here $C^{k,1}(\overline{B'_r})$ denotes the space of functions of class $C^k$
on $\overline{B'_r}$ all of whose partial derivatives of order $k$ are
Lipschitz continuous. By compactness of $\partial\Omega$, finitely many charts
\eqref{eq:local-graph} cover $\partial\Omega$, and the $C^{k,1}$ norms of the
corresponding $\varphi$'s may be taken uniformly bounded. On a compact
boundary, $C^{k+1}\subset C^{k,1}\subset C^k$; in particular a $C^3$ boundary
is $C^{2,1}$. This is the standard local-graph definition; see e.g.\ \cite[Def.~1.2.1.1]{Grisvard85} or \cite[\S6.2]{GilbargTrudinger01}.
\end{definition}

\begin{remark}\label{rem:C21}
We use Definition~\ref{def:Ck1} with $k=2$. In the chart
\eqref{eq:local-graph} the outward unit normal vector is
\[
n\left(y',\varphi(y')\right)
=\frac{\left(\nabla'\varphi(y'),\,-1\right)}{\sqrt{1+|\nabla'\varphi(y')|^2}}\,,
\]
so that $\partial\Omega\in C^{2,1}$ means exactly that $n\in C^{1,1}(\partial\Omega)$;
consequently, the shape operator $D_\tau n$, and hence the second fundamental
form $\II$, is well defined and Lipschitz on $\partial\Omega$, and in
particular $\|\II\|_{L^\infty(\partial\Omega)}<\infty$, as used in
Lemma~\ref{lem:reilly}. The $C^{2,1}$ regularity is precisely what the $H^3$
elliptic estimate in the Step~5 in the proof of Lemma~\ref{lem:reilly} requires: the Neumann
problem gains two derivatives ($H^{k+2}$ regularity) on domains with
$C^{k+1,1}$ boundary, here with $k=1$; see \cite[Thm.~2.5.1.1]{Grisvard85}.
\end{remark}
\begin{definition}[Specular class]\label{def:specular}
Let $\mathscr C$ denote the set of functions $f\in C^2(\overline\Omega\times\R^d)$ such that
\begin{equation}\label{eq:poly-growth}
|\partial^\alpha f(x,v)|\le C_\alpha(1+|v|)^{N}
\quad\text{for all }(x,v)\in\overline\Omega\times\R^d\text{ and all multi-indices }\alpha\text{ with }|\alpha|\le2,
\end{equation}
where $\partial^\alpha$ ranges over all mixed derivatives in $(x,v)$, for some $N\in(0,\infty)$ and constants $C_\alpha<\infty$ depending on $f$, and
\begin{equation}\label{eq:specular-class}
f(x,v)=f(x,R_xv)\qquad\text{for all }x\in\partial\Omega,\ v\in\R^d.
\end{equation}
We write $\mathscr C_0:=\mathscr C\cap\Lmuz$.
\end{definition}
The infinitesimal generator of \eqref{eq:sde} decomposes as
\begin{equation}\label{eq:decomp}
\Lgen=\La+\gamma\Ls,\qquad
\La:=v\cdot\nabla_x-\nabla_xU\cdot\nabla_v,\qquad
\Ls:=-v\cdot\nabla_v+\Delta_v .
\end{equation}
Denoting by $\nabla_x^*$ and $\nabla_v^*$ the formal adjoints of $\nabla_x,\nabla_v$ in $\Lmu$, we have $\Ls=-\nabla_v^*\nabla_v$; since $\Ls$ involves no $x$-derivatives, it is symmetric and nonpositive on $\mathscr C$.
Indeed,
for fixed $x$, let us integrate by parts in $v$ against the Gaussian $\rmd\kappa\propto e^{-|v|^2/2}\rmd v$. Since $\nabla_v^{*}=-\nabla_v+v$ in $L^2(\kappa)$, we get for each fixed $x\in\Omega$,
\begin{align}\label{by:integrating}
\int_{\R^d}(\Ls g)\,f\,\rmd\kappa
&=\int_{\R^d}\left(\Delta_vg-v\cdot\nabla_vg\right)f\,\rmd\kappa
\\
&=\int_{\R^d}\nabla_v\cdot\left(e^{-|v|^2/2}\nabla_vg\right)\frac{f}{e^{-|v|^2/2}}\rmd\kappa
=-\int_{\R^d}\nabla_vf\cdot\nabla_vg\,\rmd\kappa,
\nonumber
\end{align}
where the last step is the divergence theorem on $\{|v|<R\}$ together with the vanishing of 
\[
\int_{|v|=R}f\,(\nabla_vg\cdot v/|v|)\,e^{-|v|^2/2}\rmd S\to0,\qquad\text{as $R\to\infty$}, 
\]
which is where \eqref{eq:poly-growth} is used.
Finally, by integrating \eqref{by:integrating} with respect to $\mux$, we get
\begin{equation}\label{eq:Ls-sym}
\ip{f}{\Ls g}=\ip{\Ls f}{g}=-\int_{\Omega\times\mathbb{R}^{d}}\nabla_vf\cdot\nabla_vg\,\rmd\mu .
\end{equation}
The transport part requires the boundary condition:
\begin{lemma}[Antisymmetry of $\La$ on the specular class]\label{lem:antisym}
For all $f,g\in\mathscr C$,
\begin{equation}\label{eq:antisym}
\ip{\La f}{g}=-\ip{f}{\La g}.
\end{equation}
\end{lemma}
\begin{proof}
Since $\La$ is a first-order differential operator it satisfies the Leibniz rule 
\[
\La(fg)=(\La f)g+f(\La g), 
\]
and hence \eqref{eq:antisym} is equivalent to
\begin{equation}\label{eq:antisym-reduced}
\int_{\Omega\times\R^d}\La\varphi\,\rmd\mu=0,\qquad \varphi:=fg .
\end{equation}
Note that $\varphi\in\mathscr C$: it is $C^2$, it obeys \eqref{eq:poly-growth} (with $N$ doubled), and it satisfies \eqref{eq:specular-class} because $f$ and $g$ do, the condition \eqref{eq:specular-class} being stable under products.

\emph{Step 1: the $x$-integration by parts.} Fix $v\in\R^d$ and apply the divergence theorem to the vector field $x\mapsto \varphi(x,v)\,e^{-U(x)}\,v$ on $\Omega$, whose divergence is $e^{-U}\left(v\cdot\nabla_x\varphi-\varphi\,v\cdot\nabla_xU\right)$. With $n$ the \emph{outward} normal,
\[
\int_\Omega \left(v\cdot\nabla_x\varphi\right)e^{-U}\rmd x
=\int_\Omega\varphi\,(v\cdot\nabla_xU)\,e^{-U}\rmd x
+\int_{\partial\Omega}\varphi(x,v)\,(v\cdot n(x))\,e^{-U(x)}\rmd\sigma(x),
\]
where $\rmd\sigma$ is the $(d-1)$-dimensional surface measure on $\partial\Omega$. Multiplying by $Z^{-1}e^{-|v|^2/2}$ and integrating in $v$ (legitimate by \eqref{eq:poly-growth} and Fubini's theorem, the integrand being dominated by $C(1+|v|)^{2N+1}e^{-|v|^2/2}$ up to a constant) gives
\begin{equation}\label{eq:ibp-x}
\int_{\Omega\times\R^d}v\cdot\nabla_x\varphi\,\rmd\mu
=\int_{\Omega\times\R^d}\varphi\,(v\cdot\nabla_xU)\,\rmd\mu
+\frac1Z\int_{\partial\Omega}\!\int_{\R^d}\varphi(x,v)\,(v\cdot n(x))\,e^{-H(x,v)}\rmd v\,\rmd\sigma(x).
\end{equation}

\emph{Step 2: the $v$-integration by parts.} For fixed $x$, using $\nabla_v^*=-\nabla_v+v$ in $L^2(\kappa)$ and $\nabla_xU$ is constant in $v$,
\begin{equation}\label{eq:ibp-v}
-\int_{\Omega\times\R^d}\nabla_xU\cdot\nabla_v\varphi\,\rmd\mu
=-\int_{\Omega\times\R^d}\varphi\,(v\cdot\nabla_xU)\,\rmd\mu ,
\end{equation}
with no boundary term, $\R^d$ having no boundary and the Gaussian decaying.

\emph{Step 3: cancellation.} Adding \eqref{eq:ibp-x} and \eqref{eq:ibp-v}, the two interior terms $\pm\int_{\Omega\times\mathbb{R}^{d}}\varphi\,(v\cdot\nabla_xU)\rmd\mu$ cancel and we are left with
\begin{equation}\label{eq:boundary-flux}
\int_{\Omega\times\R^d}\La\varphi\,\rmd\mu=\frac1Z\int_{\partial\Omega}I(x)\,\rmd\sigma(x),
\qquad
I(x):=\int_{\R^d}\varphi(x,v)\,(v\cdot n(x))\,e^{-H(x,v)}\,\rmd v .
\end{equation}

\emph{Step 4: the boundary flux vanishes.} Fix $x\in\partial\Omega$ and perform in $I(x)$ the change of variables $v=R_xw$, i.e.\ $w=R_x^{-1}v=R_xv$ by \eqref{eq:R-properties}. We check the three factors of the integrand one at a time.
\begin{itemize}
\item \emph{The measure.} $R_x$ is linear with $|\det R_x|=1$, so $\rmd v=|\det R_x|\,\rmd w=\rmd w$.
\item \emph{The weight.} $H(x,R_xw)=\frac12|R_xw|^2+U(x)=\frac12|w|^2+U(x)=H(x,w)$ by \eqref{eq:R-properties}, so $e^{-H(x,R_xw)}=e^{-H(x,w)}$.
\item \emph{The test function.} $\varphi(x,R_xw)=\varphi(x,w)$ by the specular condition \eqref{eq:specular-class}, which applies precisely because $x\in\partial\Omega$.
\item \emph{The flux factor.} $(R_xw)\cdot n(x)=w\cdot n(x)-2(w\cdot n(x))|n(x)|^2=-\,w\cdot n(x)$ by \eqref{eq:R-properties}.
\end{itemize}
Hence, we conclude that
\begin{align*}
I(x)&=\int_{\R^d}\varphi(x,R_xw)\,\left((R_xw)\cdot n(x)\right)\,e^{-H(x,R_xw)}\,\rmd w
\\
&=\int_{\R^d}\varphi(x,w)\,\left(-w\cdot n(x)\right)\,e^{-H(x,w)}\,\rmd w
=-I(x).
\end{align*}
The integral $I(x)$ is absolutely convergent by \eqref{eq:poly-growth}, hence finite, and therefore $I(x)=0$ for every $x\in\partial\Omega$. By \eqref{eq:boundary-flux} this proves \eqref{eq:antisym-reduced}, and with it the lemma.
\end{proof}
Lemma~\ref{lem:antisym} is stated for $f,g\in\mathscr C$, i.e.\ for $C^2$ functions. But in the proof of Lemma~\ref{lem:corrector} and in Section~\ref{sec:proof} it is applied to pairs such as $(\phi,\La u)$ where $u=\Am\phi\in D(\Lo)\subset H^2(\Omega)$ only; $\La u=v\cdot\nabla_xu$ then lies in $L^2$ with $\nabla_xu\in H^1(\Omega)$, and $\La u$ is \emph{not} $C^2$.
To solve this issue, we introduce
the following lemma
at Sobolev regularity.

\begin{lemma}[Antisymmetry against affine-in-$v$ Sobolev fields]\label{lem:antisym-sob}
Let $f\in\mathscr C$ and let
\[
g(x,v)=g_0(x)+v\cdot G(x),\qquad g_0\in H^1(\Omega),\quad G\in H^1(\Omega;\R^d),
\]
and assume the trace condition
\begin{equation}\label{eq:sob-specular}
n(x)\cdot G(x)=0\qquad\text{for $\sigma$-a.e.\ }x\in\partial\Omega ,
\end{equation}
where $G$ on $\partial\Omega$ means the trace of $G$. Then $g\in\Lmu$, $\La g\in\Lmu$, and
\begin{equation}\label{eq:antisym-sob}
\ip{\La f}{g}=-\ip{f}{\La g}.
\end{equation}
\end{lemma}

In Lemma~\ref{lem:antisym-sob}, Condition \eqref{eq:sob-specular} is exactly the specular symmetry of $g$, in the trace sense: for $\sigma$-a.e.\ $x\in\partial\Omega$ and every $v\in\mathbb{R}^{d}$,
\[
g(x,R_xv)-g(x,v)=-2\,(v\cdot n(x))\,\left(n(x)\cdot G(x)\right).
\]

\begin{proof}[Proof of Lemma~\ref{lem:antisym-sob}]
We divide the proof into several steps.

\emph{Step 0: integrability.} Since
\[
\La g=v\cdot\nabla_xg_0+\sum_{i,j}v_iv_j\,\partial_iG_j-\nabla_xU\cdot G,
\]
and $v$-polynomials belong to every $L^p(\kappa)$, while $g_0,G\in H^1(\Omega)$ and $\nabla_xU\in C^1(\overline\Omega)$ is bounded, both $g$ and $\La g$ are in $\Lmu$; and $\La f\in\Lmu$ by \eqref{eq:poly-growth}. All the integrals below are absolutely convergent.

\emph{Step 1: Leibniz rule.} Put $\varphi:=fg$. For each fixed $v$, $f(\cdot,v)\in C^1(\overline\Omega)$ with $f(\cdot,v)$ and $\nabla_xf(\cdot,v)$ bounded, and $g(\cdot,v)\in H^1(\Omega)$; hence $\varphi(\cdot,v)\in H^1(\Omega)$, its trace is $\gamma_0\varphi=f\,\gamma_0g$, and the product rule for weak derivatives gives
\[
\La\varphi=(\La f)g+f\,\La g\qquad\text{a.e. on $\Omega\times\R^d$}.
\]
Thus, \eqref{eq:antisym-sob} is again equivalent to
\[
\int_{\Omega\times\R^d}\La\varphi\,\rmd\mu=0.
\]

\emph{Step 2: the $x$-integration by parts.} For fixed $v$, the field $x\mapsto\varphi(x,v)e^{-U(x)}v$ belongs to $H^1(\Omega;\R^d)$, and $\Omega$ is a bounded Lipschitz domain, so the Gauss--Green formula applies in the trace sense (see e.g.\ \cite[\S1.5.3]{Grisvard85}) and yields \eqref{eq:ibp-x} verbatim, with $\varphi(x,v)$ on $\partial\Omega$ read as $\gamma_0\varphi(\cdot,v)(x)$. Multiplying by $Z^{-1}e^{-|v|^2/2}$ and integrating in $v$ is legitimate by Fubini's theorem, because by the trace theorem
\[
\int_{\R^d}\!\!\int_{\partial\Omega}|\gamma_0\varphi(x,v)|\,|v|\,e^{-H}\rmd\sigma\,\rmd v
\le C\!\int_{\R^d}\!(1+|v|)^{N+2}e^{-|v|^2/2}\rmd v\;\left(\|g_0\|_{H^1(\Omega)}+\|G\|_{H^1(\Omega)}\right)<\infty .
\]

\emph{Step 3: the $v$-integration by parts.} Identical to Step~2 of Lemma~\ref{lem:antisym}: for a.e.\ fixed $x$, $\varphi(x,\cdot)$ is a $C^2$ function of $v$ with polynomial growth, so \eqref{eq:ibp-v} holds.

\emph{Step 4: the boundary flux vanishes.} As in \eqref{eq:boundary-flux}, what remains is $Z^{-1}\int_{\partial\Omega}I(x)\rmd\sigma(x)$ with
\[
I(x)=\int_{\R^d}\gamma_0\varphi(x,v)(v\cdot n(x))e^{-H(x,v)}\rmd v.
\]
Fix $x\in\partial\Omega$ in the full-measure set where \eqref{eq:sob-specular} holds. Then
\[
\gamma_0\varphi(x,R_xw)=f(x,R_xw)\,g(x,R_xw)=f(x,w)\,g(x,w)=\gamma_0\varphi(x,w),
\]
using \eqref{eq:specular-class} for $f$ and \eqref{eq:sob-specular} for $g$. The measure, the weight and the flux factor transform under $v=R_xw$ exactly as in Step~4 of Lemma~\ref{lem:antisym}, so $I(x)=-I(x)$, i.e.\ $I(x)=0$ for $\sigma$-a.e.\ $x$. Hence,
\[
\int_{\Omega\times\R^d}\La\varphi\,\rmd\mu=0.
\]
This completes the proof.
\end{proof}
Consequently, for $f\in\mathscr C$,
\[
\ip{f}{\Lgen f}=\ip{f}{\La f}+\gamma\ip{f}{\Ls f},
\]
and
\[
\ip{f}{\La f}=-\ip{\La f}{f}=-\ip{f}{\La f},
\]
by Lemma~\ref{lem:antisym}, so that $\ip{f}{\La f}=0$; then apply \eqref{eq:Ls-sym} with $g=f$, we get
\begin{equation}\label{eq:energy}
\ip{f}{\Lgen f}=\gamma\,\ip{f}{\Ls f}=-\gamma\,\nrm{\nabla_vf}^2 ,
\end{equation}
which vanishes on $\ker\Ls=\{f:\nabla_vf=0\}$; as in the whole space, $\Lgen$ is not coercive and the decay must be extracted hypocoercively.
\subsection{Velocity averaging and the Neumann overdamped infinitesimal generator}
Let $\Pv$ denote the orthogonal projection onto $\ker\Ls$,
\begin{equation}\label{eq:Pv}
(\Pv f)(x):=\int_{\R^d}f(x,v)\,\rmd\kappa(v),
\end{equation}
which clearly preserves the specular class (its output is independent of $v$, so \eqref{eq:specular-class} holds trivially). We claim the following identities hold:
\begin{equation}\label{eq:LaPv}
\La\Pv f=v\cdot\nabla_x(\Pv f),\qquad
\Pv\La\Pv=0.
\end{equation}
For the first identity in \eqref{eq:LaPv}, since $\Pv f$ depends on $x$ only, we can compute that $\nabla_v(\Pv f)=0$ and
\[
\La\Pv f=v\cdot\nabla_x(\Pv f)-\nabla_xU\cdot\nabla_v(\Pv f)=v\cdot\nabla_x(\Pv f).
\]
For the second identity in \eqref{eq:LaPv}, we can compute that
\[
\Pv\La\Pv f=\int_{\R^d}v\cdot\nabla_x(\Pv f)\rmd\kappa(v)=\left(\int_{\R^d}v\,\rmd\kappa\right)\cdot\nabla_x(\Pv f)=0,
\]
where we used $\int_{\mathbb{R}^{d}} v\,\rmd\kappa=0$.
Next, since $\Pv$ is an orthogonal projection in $\Lmu$, it is self-adjoint, so that for $f,g\in\mathscr C$,
\begin{align}\label{two:equalities}
\ip{\La\Pv f}{g}=-\ip{\Pv f}{\La g}=-\ip{f}{\Pv\La g},
\end{align}
where the first equality in \eqref{two:equalities} is due to the fact that $\Pv f$ is $v$-independent and hence specular
so that one can apply Lemma~\ref{lem:antisym} to the pair $(\Pv f,g)$,
and the second equality in \eqref{two:equalities} is due to the self-adjointness of $\Pv$.
Therefore, for $f,g\in\mathscr C$,
\begin{equation}\label{eq:adjoint-LaPv}
(\La\Pv)^*=-\Pv\La
\quad\text{on }\mathscr C.
\end{equation}
Next, we define the overdamped infinitesimal generator through its Dirichlet form. Let
\[
\mathcal E(g,h):=\int_\Omega\nabla_xg\cdot\nabla_xh\,\rmd\mux,\qquad g,h\in H^1(\mux)=H^1(\Omega),
\]
(the equality of spaces holds since $e^{-U}$ is bounded above and below on $\overline\Omega$, because $U\in C^2(\overline\Omega)$ and $\overline\Omega$ is compact), and let $-\Lo$ be the nonnegative self-adjoint operator on $L^2(\mux)$ associated with $\mathcal E$. This is the \emph{Neumann realization}, which is summarized
in Lemma~\ref{lem:neumann} below. We first record, once and for all, the weighted Green formula that will be used repeatedly, at the Sobolev generality that is needed later.

\begin{lemma}[Weighted Green formula]\label{lem:green}
Let $\Omega\subset\R^d$ be a bounded Lipschitz domain and $U\in C^1(\overline\Omega)$, and write $Z_x:=\int_\Omega e^{-U}\rmd x$. For all $a\in H^1(\Omega)$ and $b\in H^2(\Omega)$,
\begin{equation}\label{eq:green}
\int_\Omega\nabla_xa\cdot\nabla_xb\,\rmd\mux=-\int_\Omega a\,\left(\Delta_xb-\nabla_xU\cdot\nabla_xb\right)\,\rmd\mux+\frac1{Z_x}\int_{\partial\Omega}a\,\partial_nb\,e^{-U}\rmd\sigma .
\end{equation}
\end{lemma}

\begin{proof}
For $a,b\in C^2(\overline\Omega)$ this is the divergence theorem applied to the vector field $a\,e^{-U}\nabla_xb$. In general, pick $a_k,b_k\in C^2(\overline\Omega)$ with $a_k\to a$ in $H^1(\Omega)$ and $b_k\to b$ in $H^2(\Omega)$, which is possible on a bounded Lipschitz domain; every term in \eqref{eq:green} passes to the limit --- the interior integrals by the Cauchy--Schwarz inequality, the boundary integral by the trace theorem, since $a_k\to a$ and $\partial_nb_k\to\partial_nb$ in $L^2(\partial\Omega)$.
\end{proof}

The Neumann realization is presented in the following lemma.

\begin{lemma}[The Neumann realization of $\Lo$]\label{lem:neumann}
Let $\Omega\subset\R^d$ be a bounded convex domain and $U\in C^2(\overline\Omega)$, and let $-\Lo$ be the nonnegative self-adjoint operator on $L^2(\mux)$ associated with the closed form $\mathcal E$ on $H^1(\Omega)$. Then,
\begin{align}\label{eq:neumann-domain}
D(\Lo)=\{h\in H^2(\Omega):\ \partial_nh=0\ \text{$\sigma$-a.e.\ on }\partial\Omega\},\qquad
\Lo h=\Delta_xh-\nabla_xU\cdot\nabla_xh,
\end{align}
and there is $C=C(\Omega,U)$ such that
\begin{equation}\label{eq:apriori-H2}
\|h\|_{H^2(\Omega)}\le C\left(\nrm[\mux]{\Lo h}+\nrm[\mux]{h}\right),\qquad h\in D(\Lo).
\end{equation}
\end{lemma}
\begin{proof}
\emph{Inclusion $\subseteq$.} Let $h\in D(\Lo)$. By definition of the form, $h\in H^1(\Omega)$ and
$\mathcal E(h,\psi)=\ip[\mux]{-\Lo h}{\psi}$ for all $\psi\in H^1(\Omega)$; unwinding the weight, this says that $h$ is a weak solution of the Neumann problem
\[
\Delta_x h=F:=\Lo h+\nabla_xU\cdot\nabla_xh\ \text{ in }\Omega,\qquad \partial_nh=0\ \text{ on }\partial\Omega .
\]
Here $F\in L^2(\Omega)$ because $\Lo h\in L^2$ and $\nabla_xU\in C^1(\overline\Omega)$ is bounded while $\nabla_xh\in L^2$. Since $\Omega$ is bounded and convex, the variational solution of the Neumann problem with $L^2$ data lies in $H^2(\Omega)$ and satisfies
\[
\|h\|_{H^2}\le C(\|\Delta h\|_{L^2}+\|h\|_{L^2});
\]
see \cite[Thm.~3.2.1.3]{Grisvard85}.  Hence $h\in H^2(\Omega)$ with $\partial_nh=0$ and $\Lo h=\Delta_xh-\nabla_xU\cdot\nabla_xh$.

Next, let us prove the estimate \eqref{eq:apriori-H2}. From the above discussions,
\[
\|h\|_{H^2}\le C(\|\Lo h\|_{L^2}+\|\nabla_xU\|_\infty\|\nabla_xh\|_{L^2}+\|h\|_{L^2}),
\]
and the first-order term is absorbed because $h\in D(\Lo)$ gives
\begin{align*}
\nrm[\mux]{\nabla_xh}^2=\mathcal E(h,h)
&=\ip[\mux]{-\Lo h}{h}
\\
&\le\nrm[\mux]{\Lo h}\,\nrm[\mux]{h}
\le\frac12\left(\nrm[\mux]{\Lo h}+\nrm[\mux]{h}\right)^2 .
\end{align*}
Since $e^{-U}$ is bounded above and below on $\overline\Omega$, the weighted and unweighted $L^2$ norms are equivalent, and \eqref{eq:apriori-H2} follows.

\emph{Inclusion $\supseteq$.} If $h\in H^2(\Omega)$ with $\partial_nh=0$, then $h\in H^1(\Omega)$ is in the form domain, and the weighted Green formula \eqref{eq:green} (Lemma~\ref{lem:green}) with $a=\psi\in H^1(\Omega)$, $b=h$ gives 
\[
\mathcal E(\psi,h)=\ip[\mux]{\psi}{-\Lo h}, 
\]
with no boundary term, so that $h\in D(\Lo)$ with $\Lo h=\Delta_xh-\nabla_xU\cdot\nabla_xh$.
\end{proof}

The Poincar\'e inequality \eqref{eq:poincare} states exactly that $-\Lo$ has a spectral gap $m$ on $L^2_0(\mux)$:
\begin{equation}\label{eq:gap}
\Spec\left(-\Lo|_{L^2_0(\mux)}\right)\subset[m,\infty).
\end{equation}
To see \eqref{eq:poincare} is equivalent to \eqref{eq:gap}, note that by the spectral theorem for the nonnegative self-adjoint operator $-\Lo$ restricted to $L^2_0(\mux)$, one has
\[
\inf\Spec\left(-\Lo|_{L^2_0(\mux)}\right)=\inf\left\{\frac{\mathcal E(g,g)}{\|g\|^2_{L^2(\mux)}}: g\in H^1(\mux)\cap L^2_0(\mux),\ g\ne0\right\},
\]
and $\mathcal E(g,g)=\|\nabla_xg\|^2_{L^2(\mux)}$.
The Poincar\'e inequality says this infimum is greater than or equal to $m$. Note that $L^2_0(\mux)$ is exactly the orthogonal complement of $\ker(-\Lo)=\mathrm{span}\{1\}$, the kernel being one-dimensional because $\Omega$ is connected (automatic, $\Omega$ being convex).
The link between $\La$, $\Pv$ and $\Lo$ is the same as in the whole space, now understood as an identity of quadratic forms on $H^1(\Omega)$ rather than of operators on a domain.
For $g,h\in H^1(\Omega)$, identified with their trivial lifts $g(x,v):=g(x)$, using $\int_{\mathbb{R}^{d}} v_iv_j\,\rmd\kappa=\delta_{ij}$,
\begin{equation}\label{eq:form-identity}
\ip{\La g}{\La h}=\int_\Omega\nabla_xg\cdot\nabla_xh\,\rmd\mux=\mathcal E(g,h),
\end{equation}
so that, in operator form on $\Ran\Pv$,
\begin{equation}\label{eq:LaPv-star-LaPv}
(\La\Pv)^*\La\Pv=-\Lo\,\Pv ,
\end{equation}
with $\Lo$ the Neumann realization \eqref{eq:neumann-domain}. No boundary condition is imposed in \eqref{eq:form-identity} --- it is a form identity on all of $H^1$ --- and this is precisely why the Dolbeault--Mouhot--Schmeiser construction automatically selects the Neumann operator: $\Lo$ is \emph{defined} through \eqref{eq:LaPv-star-LaPv}.
\begin{remark}\label{rem:neumann-specular}
The Neumann boundary condition in \eqref{eq:neumann-domain} interacts with the specular class in an essential way: for $h\in D(\Lo)$, the function $\La h=v\cdot\nabla_xh$ satisfies, at $x\in\partial\Omega$,
\[
(R_xv)\cdot\nabla_xh=v\cdot\nabla_xh-2(v\cdot n)\,(n\cdot\nabla_xh)=v\cdot\nabla_xh-2(v\cdot n)\,\partial_nh=v\cdot\nabla_xh,
\]
i.e.\ $\La h$ satisfies the specular symmetry \eqref{eq:specular-class} (in the trace sense). All auxiliary functions appearing in the method below are of this form, so Lemma~\ref{lem:antisym-sob} applies to every integration by parts we perform.
Note that the implication is an \emph{equivalence}: since the identity must hold for all $v$, and $v\mapsto(v\cdot n)$ is not identically zero, $\La h$ is specular if and only if $\partial_nh=0$.
\end{remark}
\subsection{The gap-shifted corrector and the modified functional}
Following \cite{FLL26}, define the gap-shifted corrector and the modified Dolbeault--Mouhot--Schmeiser functional
\begin{equation}\label{eq:Am-def}
\begin{aligned}
&\Am:=(m-\Lo)^{-1}(\La\Pv)^*=-(m-\Lo)^{-1}\Pv\La,
\\
&\mathsf L_m(f):=\frac12\nrm{f}^2-\varepsilon\ip{\Am f}{f},
\end{aligned}
\end{equation}
with $\varepsilon>0$ to be chosen. Since $(\La\Pv)^*$ takes values in $\Ran\Pv$ and $(m-\Lo)^{-1}$ acts in the $x$-variable, $\Am f$ is a function of $x$ alone (hence specular), and moreover $\Am f\in D(\Lo)$ satisfies the Neumann condition. Next, we claim that
\begin{subequations}\label{eq:Am-identities}
\begin{align}
\Am&=\Pv\Am,\label{eq:Am-a}\\
\Am\Pv&=0,\label{eq:Am-b}\\
\Am\La\Pv&=(m-\Lo)^{-1}(-\Lo)\Pv.\label{eq:Am-c}
\end{align}
\end{subequations}
Indeed, the range of $(m-\Lo)^{-1}$ consists of functions of $x$ alone, and $\Pv$ acts as the identity on such functions,
which implies \eqref{eq:Am-a}. By the second identity in \eqref{eq:LaPv},
\[
\Am\Pv=-(m-\Lo)^{-1}\Pv\La\Pv=0,
\]
which implies \eqref{eq:Am-b}.
Finally, by \eqref{eq:LaPv-star-LaPv},
\[
\Am\La\Pv=(m-\Lo)^{-1}(\La\Pv)^*\La\Pv=(m-\Lo)^{-1}(-\Lo)\Pv,
\]
which implies \eqref{eq:Am-c}. Note that \eqref{eq:Am-c} is an operator identity on $\Ran\Pv\cap D(\Lo)$ and extends to $\Ran\Pv$ by boundedness of $(m-\Lo)^{-1}(-\Lo)=\mathrm{Id}-m(m-\Lo)^{-1}$.
The mechanism of the gap shift is unchanged: on the slow component $f_S=\Pv f$, an eigenfunction of $-\Lo$ (Neumann) with eigenvalue $\lambda\ge m$ acquires from \eqref{eq:Am-c} the coercive prefactor $\lambda/(m+\lambda)\ge\frac12$, so even the slowest macroscopic (Neumann) mode contributes at order one. By contrast, the unshifted corrector $(1-\Lo)^{-1}(\La\Pv)^*$ of \cite{DMS09,DMS15} produces the prefactor $\lambda/(1+\lambda)$, which on the slowest mode $\lambda\simeq m$ is only of order $m$; this is the source of the loss of powers of $m$ documented in \cite[Prop.~B.2]{CLW23}.
\subsection{From the reflected SDE to the boundary-value problem}\label{sec:sde-pde}
It is convenient to record the exact relationship between the pathwise description \eqref{eq:sde} of the process and the boundary-value problem \eqref{eq:bke}. Write the velocity reflection as a c\`adl\`ag correction: with $b(x,v):=-\nabla U(x)-\gamma v$ and $\sigma:=\sqrt{2\gamma}$, the system \eqref{eq:sde} is
\begin{equation}\label{eq:sde-jump}
X_t=X_0+\int_0^tV_s\,\rmd s,
\qquad
V_t=V_0+\int_0^tb(X_s,V_s)\,\rmd s+\sigma W_t+\mathcal R^S_t,
\end{equation}
where
\begin{equation}\label{eq:jump-term}
\mathcal R^S_t:=-\sum_{0<s\le t}2\left(V_{s-}\cdot n(X_s)\right)\,n(X_s)\,\mathbf 1_{\partial\Omega}(X_s),
\qquad t\ge0 .
\end{equation}
Thus $V$ is continuous off the collision times $\mathcal T:=\{s>0:X_s\in\partial\Omega\}$ and, at each $s\in\mathcal T$,
\[
\Delta V_s:=V_s-V_{s-}=-2\left(V_{s-}\cdot n(X_s)\right)n(X_s),
\qquad\text{i.e.}\qquad
V_s=R_{X_s}V_{s-},
\]
which is exactly \eqref{eq:sde-reflect} with the reflection map \eqref{eq:reflection-map}. Note that $|V_s|=|V_{s-}|$ by \eqref{eq:R-properties}: the collision is elastic and the Hamiltonian $H$ is conserved across it. We also introduce the \emph{grazing set}
\begin{equation}\label{eq:grazing}
\Gamma_0:=\{(x,v)\in\partial\Omega\times\R^d:\ v\cdot n(x)=0\},
\end{equation}
the set of boundary states at which the reflection \eqref{eq:reflection-map} acts trivially.
The following proposition gives a martingale characterization of the solution $(X,V)$ to \eqref{eq:sde-jump}--\eqref{eq:jump-term}.

\begin{proposition}[The specular class is exactly the class with no jumps]\label{prop:generator}
Let $f\in\mathscr C$ and let $(X,V)$ solve \eqref{eq:sde-jump}--\eqref{eq:jump-term} on $[0,T]$ with $\mathcal T\cap[0,T]$ finite and $(X_s,V_{s-})\notin\Gamma_0$ for every $s\in\mathcal T$. Then
\begin{equation}\label{eq:ito-jump}
\begin{aligned}
f(X_t,V_t)&=f(X_0,V_0)+\int_0^t\Lgen f(X_s,V_s)\,\rmd s
+\sigma\int_0^t\nabla_vf(X_s,V_s)\cdot \rmd W_s
\\
&\qquad\qquad\qquad+\sum_{0<s\le t}\Big[f\left(X_s,R_{X_s}V_{s-}\right)-f\left(X_s,V_{s-}\right)\Big],
\end{aligned}
\end{equation}
with $\Lgen$ the operator in \eqref{eq:bke}. In particular, the jump sum vanishes identically --- so that $f(X_t,V_t)-\int_0^t\Lgen f(X_s,V_s)\,\rmd s$ is a martingale and $\Lgen$ is the infinitesimal generator with no boundary contribution --- \emph{if and only if} $f$ satisfies the specular symmetry \eqref{eq:specular-class}.
\end{proposition}

\begin{proof}
Between consecutive collision times the paths of $(X,V)$ are continuous semimartingales and $\mathcal R^S$ is constant, so the classical It\^o formula applies on each such interval and produces
\[
\rmd f=\left(v\cdot\nabla_xf+b\cdot\nabla_vf+\frac{\sigma^2}{2}\Delta_vf\right)\rmd s+\sigma\nabla_vf\cdot\rmd W_s
=\Lgen f\,\rmd s+\sigma\nabla_vf\cdot\rmd W_s ,
\]
using $\sigma^2/2=\gamma$ and $b=-\nabla U-\gamma v$, which is precisely \eqref{eq:bke}. Summing over the (finitely many) intervals and adding the increments of $f$ across the collision times gives \eqref{eq:ito-jump}, the jump of $f$ at $s\in\mathcal T$ being
\[
f(X_s,V_s)-f(X_s,V_{s-})=f(X_s,R_{X_s}V_{s-})-f(X_s,V_{s-}),
\]
since $X$ is continuous.
If $f$ satisfies \eqref{eq:specular-class} then each bracket in \eqref{eq:ito-jump} vanishes because $X_s\in\partial\Omega$. 

Conversely, suppose the jump sum vanishes almost surely for every initial condition. Fix $x_0\in\partial\Omega$ and $v_0$ with $v_0\cdot n(x_0)>0$, and for small $\epsilon>0$ start the process deterministically from $(X_0,V_0)=(x_0-\epsilon v_0,\,v_0)$. On the event that the drift and noise contributions to the velocity remain small on $[0,2\epsilon]$ --- an event of probability tending to one as $\epsilon\downarrow0$ --- the trajectory follows the ray $t\mapsto x_0+(t-\epsilon)v_0$ up to a vanishing error and, since $v_0\cdot n(x_0)>0$ makes the crossing transversal, its first collision time $\tau_\epsilon$ tends to $\epsilon$ and the first collision state satisfies $(X_{\tau_\epsilon},V_{\tau_\epsilon-})\to(x_0,v_0)$ in probability as $\epsilon\downarrow0$. Since each jump vanishes almost surely, $f(X_{\tau_\epsilon},R_{X_{\tau_\epsilon}}V_{\tau_\epsilon-})=f(X_{\tau_\epsilon},V_{\tau_\epsilon-})$, and letting $\epsilon\downarrow0$, the continuity of $f$ and of $x\mapsto n(x)$ gives $f(x_0,R_{x_0}v_0)=f(x_0,v_0)$. As $(x_0,v_0)$ was arbitrary in $\partial\Omega\times\{v\cdot n>0\}$, and as $R_{x_0}$ is an involution which fixes $\{v\cdot n=0\}$ pointwise, \eqref{eq:specular-class} holds for all $v\in\R^d$.
\end{proof}

\subsection{Well-posedness}\label{sec:wp}
The construction of the reflected semigroup and the identification of a core involve kinetic trace theory near the grazing set $\Gamma_0$ defined in \eqref{eq:grazing}, which is not the focus of this paper. We summarize what is known in the literature.

Well-posedness for \eqref{eq:sde} and for the boundary-value problem \eqref{eq:kfp}--\eqref{eq:specularBC} has been studied in the literature, and we do not reprove it here. Existence and uniqueness in law for specularly reflected kinetic SDEs of the form \eqref{eq:sde-jump}--\eqref{eq:jump-term} on bounded $C^3$ domains are established in \cite{BossyJabir15} (for bounded measurable drifts; see also \cite{BossyJabir11,Polytope13,BossyJabirMaftei17,LST24}). The kinetic Fokker--Planck equation \eqref{eq:kfp} with the specular boundary condition \eqref{eq:specularBC} is well posed in the renormalized/trace framework of \cite{Carrillo98,Mischler10}, which also provides a Green formula for $\La$ in the trace spaces $L^2(\partial\Omega\times\R^d;|v\cdot n|\,e^{-H}\rmd v\,\rmd\sigma)$. Finally, strongly continuous $L^2$ semigroups under general Maxwell boundary conditions --- including the vanishing-accommodation endpoint, i.e.\ pure specular reflection --- are constructed in \cite{BCMT23} for potential-free linear kinetic models. 

The two standard properties we actually use about the semigroup, beyond the existence of a core, can be proved directly in the following  proposition.

\begin{proposition}[Invariance and contractivity]\label{prop:invariance}
Let $f,g\in\mathscr C$. Then
\begin{equation}\label{eq:mu-invariant}
\int_{\Omega\times\R^d}\Lgen f\,\rmd\mu=0,
\end{equation}
so that $\mu$ is invariant for \eqref{eq:sde} and $\Lmuz$ is preserved by the flow; and along any solution of \eqref{eq:bke} in $\mathscr C$,
\begin{equation}\label{eq:contractivity}
\frac{\rmd}{\rmd t}\,\frac12\nrm{f(t)}^2=-\gamma\nrm{\nabla_vf(t)}^2\le0 ,
\end{equation}
so the flow is an $\Lmu$-contraction. Moreover $\Lgen$ is reversible up to velocity reversal: writing $(\mathcal Sf)(x,v):=f(x,-v)$, one has $\mathcal S^2=\mathrm{Id}$, $\mathcal S\mathscr C=\mathscr C$, and
\begin{equation}\label{eq:reversibility}
\Lgen^*=\mathcal S\,\Lgen\,\mathcal S\qquad\text{on }\mathscr C,
\end{equation}
where $\Lgen^*$ is the $\Lmu$-adjoint.
\end{proposition}
\begin{proof}
\emph{Invariance.} Apply Lemma~\ref{lem:antisym} to the pair $(f,1)$: the constant function $1$ lies in $\mathscr C$, and $\La1=0$, so that
\begin{align}\label{adding:1}
\int_{\Omega\times\mathbb{R}^{d}}\La f\,\rmd\mu=\ip{\La f}{1}=-\ip{f}{\La 1}=0 .
\end{align}
For the symmetric part, \eqref{eq:Ls-sym} with $g=1$ gives
\begin{align}\label{adding:2}
\int_{\Omega\times\mathbb{R}^{d}}\Ls f\,\rmd\mu=\ip{\Ls f}{1}=-\int_{\Omega\times\mathbb{R}^{d}}\nabla_vf\cdot\nabla_v1\,\rmd\mu=0.
\end{align}
By adding \eqref{adding:1} and \eqref{adding:2}, we obtain
\[
\int_{\Omega\times\mathbb{R}^{d}}\Lgen f\rmd\mu=\int_{\Omega\times\mathbb{R}^{d}}\La f\rmd\mu+\gamma\int_{\Omega\times\mathbb{R}^{d}}\Ls f\rmd\mu=0,
\]
which is \eqref{eq:mu-invariant}. Invariance of $\mu$ and preservation of $\Lmuz$ follow since
\[
\frac{\rmd}{\rmd t}\int_{\Omega\times\mathbb{R}^{d}}f(t)\rmd\mu=\int_{\Omega\times\mathbb{R}^{d}}\Lgen f(t)\rmd\mu=0.
\]

\emph{Contractivity.} This is \eqref{eq:energy}:
\[
\frac{\rmd}{\rmd t}\frac12\nrm{f}^2=\ip{f}{\Lgen f}=\ip{f}{\La f}+\gamma\ip{f}{\Ls f},
\]
the first term vanishing by antisymmetry and the second equaling $-\gamma\nrm{\nabla_vf}^2$ by \eqref{eq:Ls-sym}.

\emph{Reversibility up to $\mathcal S$.} First, $\mathcal S$ preserves $\mathscr C$: it clearly preserves regularity and \eqref{eq:poly-growth}, and it preserves \eqref{eq:specular-class} because $R_x$ and $v\mapsto-v$ commute, $R_x(-v)=-R_xv$. Next, $\mathcal S$ is an isometric involution of $\Lmu$, because $\mu$ is even in $v$. A direct computation gives
\[
\mathcal S\La\mathcal S=-\La\qquad\text{and}\qquad
\mathcal S\Ls\mathcal S=\Ls.
\]
Indeed, $(\mathcal S\La\mathcal Sf)(x,v)$ is obtained from $\La$ by the substitution $v\mapsto-v$, which flips the sign of $v\cdot\nabla_x$ and of $\nabla_v$, hence flips the sign of both terms of $\La$, whereas $\Ls=-v\cdot\nabla_v+\Delta_v$ is even. Therefore, by \eqref{eq:antisym} and \eqref{eq:Ls-sym},
\[
\Lgen^*=\La^*+\gamma\Ls^*=-\La+\gamma\Ls=\mathcal S\La\mathcal S+\gamma\,\mathcal S\Ls\mathcal S=\mathcal S\Lgen\mathcal S ,
\]
which is \eqref{eq:reversibility}.
\end{proof}

\section{Main Results}\label{sec:main}
In this section, we present the main result of our paper. Before we proceed, let us first state our main assumption as below. 
\begin{assumption}\label{ass:WP}
The operator $\Lgen$ of \eqref{eq:bke}, endowed with the specular boundary condition, generates a strongly continuous contraction semigroup $(e^{t\Lgen})_{t\ge0}$ on $\Lmu$ preserving $\Lmuz$, the class $\mathscr C_0$ is a core for its infinitesimal generator, the generator coincides on $\mathscr C_0$ with the differential operator in \eqref{eq:bke}, and for $f_0\in D(\Lgen)$ the trajectory $f(t)=e^{t\Lgen}f_0$ satisfies $f\in C([0,\infty);D(\Lgen))\cap C^1([0,\infty);\Lmuz)$.
\end{assumption}

 Now, we are ready to present our main result, which is the following theorem providing
an explicitly computable decay rate for the
underdamped Langevin dynamics with specular reflection in \eqref{eq:sde}.

\begin{theorem}\label{thm:main}
Let $\Omega\subset\R^d$ be a bounded convex domain with $C^{2,1}$ boundary, $U\in C^2(\overline\Omega)$, and let $\mu$ be the Gibbs measure \eqref{eq:gibbs} on $\Omega\times\R^d$. Assume:
\begin{enumerate}
\item[(i)] the position marginal $\mux$ satisfies the Poincar\'e inequality
\begin{equation}\label{eq:poincare}
\nrm[\mux]{g}^2\le\frac1m\,\nrm[\mux]{\nabla_xg}^2,
\qquad\text{for all }g\in H^1(\mux)\text{ with }\textstyle\int_{\Omega} g\,\rmd\mux=0,
\end{equation}
with constant $m>0$;
\item[(ii)] the Hessian lower bound $\nabla^2U(x)\succeq-K\,\mathrm{Id}$ holds on $\Omega$ for some $K\ge0$;
\item[(iii)] the well-posedness Assumption~\ref{ass:WP} holds.
\end{enumerate}
Let $f(t)=e^{t\Lgen}f_0$ solve \eqref{eq:bke} with the specular boundary condition and $f_0\in\Lmuz$. With the friction coefficient
\[
\gamma=\sqrt{16m+2K},
\]
the following decay estimate holds for all $t\ge0$:
\begin{equation}\label{eq:main-decay}
\nrm{f(t)}\le\sqrt3\,e^{-\Lambda t}\nrm{f_0},\qquad
\Lambda=\frac16\,\frac{\sqrt m}{\sqrt{2+\frac{K}{2m}}+\sqrt{4+\frac{K}{2m}}}.
\end{equation}
In particular, when $U$ is convex on $\Omega$ (so $K=0$), the choice $\gamma=4\sqrt m$ yields
\[
\Lambda=\frac{2-\sqrt2}{12}\,\sqrt m.
\]
\end{theorem}

\begin{proof}
The proof will be provided in Section~\ref{sec:proof}.
\end{proof}

Note that since $\Omega$ is bounded and $U\in C^2(\overline\Omega)$, both hypotheses (i)--(ii) are automatically satisfied for \emph{some} finite $m>0$ and $K\ge0$; the content of Theorem~\ref{thm:main} is the explicit and sharply scaled dependence of the rate on these constants. Indeed $\Omega$ is bounded, convex (hence connected) with $C^{2,1}$ (hence Lipschitz) boundary, so the Neumann Poincar\'e inequality holds on $\Omega$ for the Lebesgue measure, and the weight $e^{-U}$ is bounded above and below on $\overline\Omega$; and $\nabla^2U$ is continuous on the compact set $\overline\Omega$, so that
\[
K:=\max\left\{0,-\min\nolimits_{x\in\overline\Omega}\lambda_{\min}\left(\nabla^2U(x)\right)\right\}<\infty.
\]
A case of particular interest is $U\equiv0$: the dynamics is then a ``stochastic billiard with friction'' in a convex domain, $K=0$, and $m=\lambda_1^{\mathrm N}(\Omega)$ is the Neumann spectral gap of $\Omega$, for which the Payne--Weinberger inequality \cite{PW60} (see \cite{Bebendorf03} for a correction of the original proof) gives $m\ge\pi^2/(\diam(\Omega))^2$; see Example~\ref{ex:billiard}.
Beyond the statement in Theorem~\ref{thm:main} itself, we wish to emphasize three structural points, which we believe are of independent interest for transferring $L^2$ hypocoercivity methods to boundary-value problems.

\emph{(a) Specular symmetry and antisymmetry of the transport operator.} The natural functional class for \eqref{eq:bke} consists of functions symmetric under the boundary reflection, $f(x,v)=f(x,R_xv)$ for $x\in\partial\Omega$. On this class the transport operator $\La$ is antisymmetric in $\Lmu$: the boundary flux term produced by integration by parts is odd under the measure-preserving involution $v\mapsto R_xv$ and therefore vanishes (Lemma~\ref{lem:antisym}). This is the exact $L^2$ analogue of energy conservation at specular collisions.

\emph{(b) The corrector selects the Neumann realization.} In a domain, the resolvent $(m-\Lo)^{-1}$ in \eqref{eq:gap-corrector} requires a self-adjoint realization of $\Lo=\Delta_x-\nabla U\cdot\nabla_x$. The algebra of the Dolbeault--Mouhot--Schmeiser method itself dictates the choice: the operator $(\La\Pv)^*\La\Pv$ is by construction the operator represented by the Dirichlet form $\int_{\Omega}|\nabla_xg|^2\rmd\mux$ on $H^1(\mux)$, i.e.\ the \emph{Neumann} realization $\Lo^{\mathrm N}$. Moreover, the Neumann boundary condition satisfied by elements of $D(\Lo^{\mathrm N})$ is precisely the condition under which functions of the form $v\cdot\nabla_xh(x)$ are specularly symmetric:
\[
(R_xv)\cdot\nabla h=v\cdot\nabla h-2(v\cdot n)\,\partial_nh=v\cdot\nabla h
\quad\text{on }\partial\Omega\iff\partial_nh=0 .
\]
Hence, every auxiliary function generated by the corrector lies in the specular class, and all integrations by parts in the method proceed without boundary contributions. This is consistent with the fact that the overdamped limit of specularly reflected Langevin dynamics is the normally reflected (Neumann) diffusion; see \cite{CesbronHutridurga16} and Remark~\ref{rem:overdamped}.

\emph{(c) Reilly instead of Bochner.} The only estimate of \cite{FLL26} in which the spatial domain genuinely enters is the second-order bound $\|\nabla^2_xh\|^2\le\|\Lo h\|^2+K\|\nabla_xh\|^2$, proved there via the Bochner identity. On a domain, the weighted Bochner identity integrates to a weighted Reilly formula carrying the boundary term $\int_{\partial\Omega}\II(\nabla h,\nabla h)\,e^{-U}\rmd\sigma$ for Neumann functions $h$, where $\II$ is the second fundamental form of $\partial\Omega$. Convexity of $\Omega$ gives $\II\succeq0$, so the boundary term has a favorable sign and the flat-space inequality survives unchanged (Lemma~\ref{lem:reilly}).

As a corollary of Theorem~\ref{thm:main}, we obtain the quantitative convergence rates in $\chi^{2}$-divergence,
total variation (TV) and Wasserstein distances.
\begin{corollary}[Convergence of the law in $\chi^2$-divergence, total variation and Wasserstein distances]\label{cor:divergences}
Assume the hypotheses of Theorem~\ref{thm:main}. Let $\nu_0$ be a probability measure on $\overline\Omega\times\R^d$ which is absolutely continuous with respect to $\mu$ with density $h_0=\rmd\nu_0/\rmd\mu\in\Lmu$, i.e.\ with finite $\chi^2$-divergence
$\chi^2(\nu_0\,\|\,\mu)<\infty$,
and let $\nu_t$ denote the law of $(X_t,V_t)$ solving \eqref{eq:sde} with $(X_0,V_0)\sim\nu_0$. Write $\nu_t^x$ for the position marginal of $\nu_t$. Then for all $t\ge0$:
\begin{align}
\chi^2(\nu_t\,\|\,\mu)&\le3\,e^{-2\Lambda t}\,\chi^2(\nu_0\,\|\,\mu),\label{eq:chi2}\\
\chi^2(\nu_t^x\,\|\,\mux)&\le3\,e^{-2\Lambda t}\,\chi^2(\nu_0\,\|\,\mu),\label{eq:chi2x}\\
\mathrm{TV}(\nu_t,\mu)&\le\frac{\sqrt3}{2}\,e^{-\Lambda t}\,\sqrt{\chi^2(\nu_0\,\|\,\mu)},\label{eq:tv}\\
\mathcal W_1(\nu_t^x,\mux)&\le\frac{\sqrt3}{2}\,\diam(\Omega)\,e^{-\Lambda t}\,\sqrt{\chi^2(\nu_0\,\|\,\mu)},\label{eq:W1x}\\
\mathcal W_2(\nu_t^x,\mux)&\le\left(\frac{3}{4}\right)^{1/4}\diam(\Omega)\,e^{-\Lambda t/2}\,\left(\chi^2(\nu_0\,\|\,\mu)\right)^{1/4},\label{eq:W2x}\\
\mathcal W_1(\nu_t,\mu)&\le\sqrt3\,\left(\left(\diam(\Omega)\right)^2+d\right)^{1/2}\,e^{-\Lambda t}\,\sqrt{\chi^2(\nu_0\,\|\,\mu)}.\label{eq:W1joint}
\end{align}
\end{corollary}
\begin{proof}
\emph{Step 1: the density evolves under the adjoint semigroup.} Let $h_t:=\rmd\nu_t/\rmd\mu$, which exists for all $t$ by the computation that follows. For a bounded measurable test function $\psi$,
\[
\int_{\Omega\times\mathbb{R}^{d}}\psi\,\rmd\nu_t=\mathbb E\big[\psi(X_t,V_t)\big]=\int_{\Omega\times\mathbb{R}^{d}} \left(e^{t\Lgen}\psi\right)\,\rmd\nu_0=\ip{e^{t\Lgen}\psi}{h_0}=\ip{\psi}{(e^{t\Lgen})^*h_0},
\]
so that 
\[
h_t=\left(e^{t\Lgen}\right)^*h_0=e^{t\Lgen^*}h_0, 
\]
where $\Lgen^*$ is the $\Lmu$-adjoint of $\Lgen$. (The second equality in the above equation used the identification $\mathbb E\big[\psi(X_t,V_t)\big]=\int_{\Omega\times\mathbb{R}^{d}}e^{t\Lgen}\psi\,\rmd\nu_0$ of the $L^2$ semigroup with the transition semigroup of \eqref{eq:sde}; this follows from the martingale property of Proposition~\ref{prop:generator} on the core $\mathscr C_0$ together with Assumption~\ref{ass:WP}, and we take it for granted in what follows.) Since $\mu$ is invariant, $e^{t\Lgen}1=1$, and hence for $h\in\Lmuz$,
\[
\int_{\Omega\times\mathbb{R}^{d}} e^{t\Lgen^*}h\,\rmd\mu=\ip{e^{t\Lgen^*}h}{1}=\ip{h}{e^{t\Lgen}1}=\ip{h}{1}=0,
\]
so $\Lmuz$ is invariant under the adjoint semigroup as well. Applying this with $h=h_0-1\in\Lmuz$ (it has mean zero because $\nu_0$ and $\mu$ are both probability measures) gives
\begin{equation}\label{eq:density-adjoint}
h_t-1=e^{t\Lgen^*}(h_0-1).
\end{equation}

\emph{Step 2: the adjoint semigroup obeys the same bound.} For any bounded operator $T$ on a Hilbert space, $\|T^*\|=\|T\|$. Applying this to $T=e^{t\Lgen}$ restricted to the invariant subspace $\Lmuz$, Theorem~\ref{thm:main} gives
\[
\big\|e^{t\Lgen^*}\big\|_{\Lmuz\to\Lmuz}=\big\|e^{t\Lgen}\big\|_{\Lmuz\to\Lmuz}\le\sqrt3\,e^{-\Lambda t}.
\]
Combined with \eqref{eq:density-adjoint} this yields
\[
\chi^2(\nu_t\|\mu)=\nrm{h_t-1}^2\le3e^{-2\Lambda t}\nrm{h_0-1}^2=3e^{-2\Lambda t}\chi^2(\nu_0\|\mu),
\]
which is \eqref{eq:chi2}.

\emph{Step 3: the position marginal.} Since $\mu=\mux\otimes\kappa$ is a product measure, the density of $\nu_t^x$ with respect to $\mux$ is exactly the velocity average of $h_t$: for bounded $\psi=\psi(x)$,
\[
\int_{\Omega}\psi\,\rmd\nu_t^x=\int_{\Omega\times\mathbb{R}^{d}}\psi\,h_t\,\rmd\mu=\int_\Omega\psi(x)\left(\int_{\R^d}h_t(x,v)\rmd\kappa(v)\right)\rmd\mux(x),
\]
so that $\rmd\nu_t^x/\rmd\mux=\Pv h_t$. As $\Pv$ is an orthogonal projection in $\Lmu$ and $\Pv1=1$,
\[
\chi^2(\nu_t^x\|\mux)=\nrm[\mux]{\Pv h_t-1}^2=\nrm{\Pv(h_t-1)}^2\le\nrm{h_t-1}^2=\chi^2(\nu_t\|\mu),
\]
which with \eqref{eq:chi2} gives \eqref{eq:chi2x}. (This is the data-processing inequality for $\chi^2$, in the special case where the marginalization is an orthogonal projection.)

\emph{Step 4: total variation.} By Cauchy--Schwarz inequality in $\Lmu$,
\[
\mathrm{TV}(\nu_t,\mu)=\frac12\int_{\Omega\times\mathbb{R}^{d}}|h_t-1|\,\rmd\mu\le\frac12\left(\int_{\Omega\times\mathbb{R}^{d}}|h_t-1|^2\rmd\mu\right)^{1/2}=\frac12\sqrt{\chi^2(\nu_t\|\mu)},
\]
and \eqref{eq:chi2} gives \eqref{eq:tv}. The same computation applied to $\nu_t^x$ and using \eqref{eq:chi2x} gives
\[
\mathrm{TV}(\nu^x_t,\mux)\le\frac{\sqrt3}{2}e^{-\Lambda t}\sqrt{\chi^2(\nu_0\|\mu)}.
\]

\emph{Step 5: Wasserstein distances for the position marginal.} Let $(X,Y)$ be a \emph{maximal coupling} of $\nu_t^x$ and $\mux$, so that $\mathbb P(X\ne Y)=\mathrm{TV}(\nu_t^x,\mux)$; such a coupling exists by the coupling characterization of the total variation distance, see e.g.\ \cite[Ch.~I]{Lindvall02} or \cite[Ch.~6]{VillaniOT09}. Since both marginals are supported in $\overline\Omega$, we have $|X-Y|\le\diam(\Omega)$ almost surely, and $|X-Y|=0$ on $\{X=Y\}$. Therefore, for $p\in\{1,2\}$,
\[
\left(\mathcal W_p(\nu_t^x,\mux)\right)^p\le\mathbb E|X-Y|^p\le\left(\diam(\Omega)\right)^p\,\mathbb P(X\ne Y)=\left(\diam(\Omega)\right)^p\,\mathrm{TV}(\nu_t^x,\mux).
\]
Taking $p=1$ and inserting Step~4 gives \eqref{eq:W1x}; taking $p=2$ and a square root gives
\[
\mathcal W_2(\nu_t^x,\mux)\le\diam(\Omega)\,\left(\mathrm{TV}(\nu_t^x,\mux)\right)^{1/2}\le\diam(\Omega)\left(\frac{\sqrt3}{2}\right)^{1/2}e^{-\Lambda t/2}\left(\chi^2(\nu_0\|\mu)\right)^{1/4},
\]
which yields \eqref{eq:W2x}.

\emph{Step 6: the joint law.} Here the velocity variable is unbounded, so we use the Kantorovich--Rubinstein duality \cite[Ch.~5]{VillaniOT09} directly. Fix $z_0=(x_0,0)$ with $x_0\in\Omega$. For $\psi$ $1$-Lipschitz on $\R^{2d}$ we may subtract the constant $\psi(z_0)$ without changing $\int_{\Omega\times\mathbb{R}^{d}}\psi\,\rmd(\nu_t-\mu)$, and then $|\psi(z)|\le|z-z_0|$. Hence, by Cauchy--Schwarz inequality,
\begin{align*}
\mathcal W_1(\nu_t,\mu)
&=\sup_{\mathrm{Lip}(\psi)\le1}\int_{\Omega\times\mathbb{R}^{d}}\psi(h_t-1)\rmd\mu
\\
&\le\int_{\Omega\times\mathbb{R}^{d}}|z-z_0|\,|h_t-1|\,\rmd\mu
\le\left(\int_{\Omega\times\mathbb{R}^{d}}|z-z_0|^2\rmd\mu\right)^{1/2}\nrm{h_t-1}.
\end{align*}
Finally,
\[
\int_{\Omega\times\mathbb{R}^{d}}|z-z_0|^2\rmd\mu=\int_\Omega|x-x_0|^2\rmd\mux+\int_{\R^d}|v|^2\rmd\kappa\le\left(\diam(\Omega)\right)^2+d,
\]
and \eqref{eq:chi2} concludes.
\end{proof}
In the unconstrained case when $U$ is convex, we know the convergence rate $O(\sqrt{m})$ is optimal \cite{FLL26}.
But for the constrained case in Theorem~\ref{thm:main}, to the best of our knowledge,
it has never been investigated in the literature whether
the scaling $O(\sqrt{m})$ is optimal. By leveraging the fact that
the specularly reflected dynamics with $U\equiv0$ on an interval has \emph{explicit eigenfunctions},
we obtain the following proposition, which itself is a genuinely new contribution.
Optimality of the scaling is meant here in the worst-case sense over the class of potentials covered by Theorem~\ref{thm:main} with a given Poincar\'e constant: the family $d=1$, $U\equiv0$, $L=\pi/\sqrt m$ realizes every value of $m>0$ within this class, so that no choice of the friction can yield a rate better than $O(\sqrt m)$ uniformly over the class.
\begin{proposition}[The $\sqrt m$ scaling is optimal for the constrained problem]\label{prop:sharp}
Let $d=1$, $\Omega=(0,L)$ and $U\equiv0$, so that $K=0$ and $m=\lambda_1^{\mathrm N}(\Omega)=\pi^2/L^2$, and let Assumption~\ref{ass:WP} hold. Define the optimal $L^2$ decay rate of the specularly reflected semigroup,
\[
\Lambda_{\mathrm{opt}}(\gamma):=\sup\Big\{\Lambda\ge0:\ \exists\,C<\infty \text{ with } \big\|e^{t\Lgen}\big\|_{\Lmuz\to\Lmuz}\le Ce^{-\Lambda t}\ \text{ for all }t\ge0\Big\}.
\]
Then
\begin{equation}\label{eq:sharp-bound}
\Lambda_{\mathrm{opt}}(\gamma)\ \le\ \min\Big\{\frac m\gamma,\ 2\gamma\Big\}\qquad\text{for every }\gamma>0,
\end{equation}
and consequently,
\begin{equation}\label{eq:sharp-sup}
\sup_{\gamma>0}\Lambda_{\mathrm{opt}}(\gamma)\ \le\ \sqrt{2m},
\end{equation}
where the supremum over $\gamma$ of the right-hand side of \eqref{eq:sharp-bound} is attained at $\gamma=\sqrt{m/2}$. In particular, no choice of the friction coefficient produces a decay rate better than $O(\sqrt m)$, so that the scaling of Theorem~\ref{thm:main} cannot be improved.
\end{proposition}

\begin{proof}
Both bounds come from exact eigenfunctions computations. Let $k:=\pi/L$, so that $m=k^2$, and set
\[
f_1(x,v):=\cos\left(k\left(x+\frac v\gamma\right)\right),
\qquad
f_2(x,v):=v^2-1 .
\]
Both $f_1$ and $f_2$ are smooth, and $f_1$ is bounded together with all its derivatives while $|\partial^\alpha f_2|\le C(1+|v|)^2$, so \eqref{eq:poly-growth} holds. In $d=1$ the reflection map is $R_xv=-v$ at both endpoints, so \eqref{eq:specular-class} requires evenness in $v$ at $x\in\{0,L\}$: for $f_2$ this is clear, and for $f_1$,
\begin{align*}
&f_1(0,-v)=\cos\left(-kv/\gamma\right)=f_1(0,v),
\\
&f_1(L,-v)=\cos\left(\pi-\frac{kv}{\gamma}\right)=-\cos\frac{kv}\gamma=f_1(L,v),
\end{align*}
using $kL=\pi$. Both have $\mu$-mean zero: $\int_{\R}(v^2-1)\rmd\kappa=0$, and
\[
\int_{\R}f_1(x,v)\rmd\kappa(v)=\Re\left(e^{ikx}\!\int_\R e^{ikv/\gamma}\rmd\kappa(v)\right)=e^{-k^2/(2\gamma^2)}\cos(kx),
\]
where
\[
\int_0^L\cos(kx)\,\frac{\rmd x}{L}=0.
\]
Both $f_1$ and $f_2$ are eigenfunctions. Indeed, with $U\equiv0$, one has
\[
\Lgen=v\partial_x+\gamma\left(-v\partial_v+\partial_v^2\right).
\]
Writing $\xi:=x+v/\gamma$, so that $\partial_xf_1=-k\sin(k\xi)$, $\partial_vf_1=-\frac k\gamma\sin(k\xi)$ and $\partial_v^2f_1=-\frac{k^2}{\gamma^2}\cos(k\xi)$,
\[
\Lgen f_1=-kv\sin(k\xi)+\gamma\left(\frac{kv}{\gamma}\sin(k\xi)-\frac{k^2}{\gamma^2}\cos(k\xi)\right)
=-\frac{k^2}{\gamma}f_1=-\frac m\gamma\,f_1 ,
\]
the transport and friction terms canceling exactly, and
\[
\Lgen f_2=\gamma(-2v^2+2)=-2\gamma f_2.
\]
By Assumption~\ref{ass:WP}, $f_i\in\mathscr C_0\subset D(\Lgen)$, so that
\[
e^{t\Lgen}f_i=e^{-\lambda_it}f_i,
\]
with $\lambda_1=m/\gamma$ and $\lambda_2=2\gamma$.
Hence,
\[
\left\|e^{t\Lgen}\right\|_{\Lmuz\to\Lmuz}\ge e^{-\min(\lambda_1,\lambda_2)t},\qquad\text{for all $t$},
\]
which is \eqref{eq:sharp-bound}. Finally, $\min\{m/\gamma,2\gamma\}$ is maximized over $\gamma>0$ at $\gamma=\sqrt{m/2}$, with value $\sqrt{2m}$, giving \eqref{eq:sharp-sup}.
\end{proof}

\begin{example}[Stochastic billiard with friction]\label{ex:billiard}
Take $U\equiv0$ on a bounded convex $C^{2,1}$ domain $\Omega$, so the process is pure kinetic Brownian motion with damping, specularly reflected at $\partial\Omega$. Then $K=0$, $\mux$ is the normalized Lebesgue measure on $\Omega$, and $m=\lambda_1^{\mathrm N}(\Omega)$ is the Neumann spectral gap, bounded below by the Payne--Weinberger inequality $m\ge\pi^2/(\diam(\Omega))^2$ \cite{PW60,Bebendorf03}. Theorem~\ref{thm:main} gives, with $\gamma=4\sqrt m$,
\[
\nrm{f(t)}\le\sqrt3\,\exp\left(-\frac{2-\sqrt2}{12}\sqrt m\,t\right)\nrm{f_0}
\qquad\text{with}\quad \sqrt m\ \ge\ \frac{\pi}{\diam(\Omega)} ,
\]
and hence the completely explicit, dimension-free rate
\[
\Lambda\ \ge\ \frac{2-\sqrt2}{12}\cdot\frac{\pi}{\diam(\Omega)}.
\]
The $O(\sqrt m)$ scaling, i.e.\ mixing in time of order $\diam(\Omega)$ rather than $(\diam(\Omega))^2$ for the friction tuned as above, reflects the ballistic transport enabled by the underdamped dynamics and is the same acceleration phenomenon as in the whole space \cite{CLW23,FLL26}.
This example is also, in dimension $d=1$, exactly the setting in which the $\sqrt m$ scaling can be shown to be \emph{optimal}; see Proposition~\ref{prop:sharp}.
\end{example}
\begin{remark}[Overdamped consistency]\label{rem:overdamped}
The appearance of the Neumann overdamped infinitesimal generator in the corrector is consistent with the diffusive limit: as $\gamma\to\infty$ (after time rescaling $t\mapsto\gamma t$), the specularly reflected underdamped Langevin dynamics converges to the normally reflected overdamped diffusion
\[
\rmd X_t=-\nabla U(X_t)\rmd t+\sqrt2\,\rmd W_t-n(X_t)\ell(\rmd t),
\]
whose infinitesimal generator is precisely $\Lo$ with Neumann boundary conditions. At the level of the SDE this is the Smoluchowski--Kramers approximation for the Langevin equation with elastic reflection, established in \cite{Spiliopoulos07}; at the level of the kinetic Fokker--Planck equation it is the diffusion limit with specular boundary conditions derived in \cite{CesbronHutridurga16}. The gap-shifted corrector thus couples the kinetic dynamics to the correct macroscopic boundary behavior, which is what allows the sharp rate to survive the presence of the boundary.
\end{remark}

\section{Key Lemmas}\label{sec:lemmas}

In this section, we present the technical lemmas that will be used in the proof of our main result (Theorem~\ref{thm:main}).
Two standard estimates are taken over unchanged from the whole-space argument of \cite{FLL26}. First, the Gaussian Poincar\'e inequality in the velocity variable (see e.g.\ \cite[Prop.~2.7.7 and \S4.1]{BGL14}): for $f\in H^1(\kappa)$,
\begin{equation}\label{eq:gauss-poincare}
\nrm[\kappa]{(1-\Pv)f}^2\le\nrm[\kappa]{\nabla_vf}^2 .
\end{equation}
Second, and this is the only genuinely domain-dependent ingredient, we need the second-order estimate on the Neumann resolvent. In the whole space this follows from the Bochner identity \cite[Ch.~1, Ch.~3]{BGL14}; on a domain, the correct statement is a weighted Reilly formula \cite{Reilly77,MaDu10,KolesnikovMilman17}.
The unweighted case $U\equiv0$ is the classical Reilly formula \cite{Reilly77}; the extension to the drifted operator $\Lo=\Delta-\nabla U\cdot\nabla$ is due to \cite{MaDu10}, and a systematic treatment on weighted manifolds with boundary is given in \cite{KolesnikovMilman17}.

\begin{lemma}[Weighted Reilly formula; second-order estimate]\label{lem:reilly}
Assume $\partial\Omega\in C^{2,1}$ (or $C^3$) and $U\in C^2(\overline\Omega)$. For every $h\in D(\Lo)$ as in \eqref{eq:neumann-domain},
\begin{equation}\label{eq:reilly}
\int_\Omega(\Lo h)^2\,\rmd\mux
=\int_\Omega\|\nabla^2_xh\|_F^2\,\rmd\mux
+\int_\Omega\langle\nabla^2U\,\nabla_xh,\nabla_xh\rangle\,\rmd\mux
+\frac1{Z_x}\int_{\partial\Omega}\II(\nabla_xh,\nabla_xh)\,e^{-U}\rmd\sigma,
\end{equation}
where $\|\cdot\|_F$ is the Frobenius norm, $\II$ is the second fundamental form of $\partial\Omega$ with respect to the outward normal and $Z_x=\int_\Omega e^{-U}\rmd x$. In particular, since $\II\succeq0$ by convexity and $\nabla^2U\succeq-K\,\mathrm{Id}$,
\begin{equation}\label{eq:hessian-bound}
\nrm[\mux]{\nabla^2_xh}^2\le\nrm[\mux]{\Lo h}^2+K\,\nrm[\mux]{\nabla_xh}^2 .
\end{equation}
\end{lemma}

\begin{proof}
Assume first $h\in H^3(\Omega)$ with $\partial_nh=0$ on $\partial\Omega$; the general case $h\in D(\Lo)$ will be treated later in \emph{Step~5}. We divide the proof into several steps.

\emph{Step 1: the weighted Bochner ($\Gamma_2$) identity.} For $\Lo=\Delta-\nabla U\cdot\nabla$ on $\R^d$ (zero Ricci curvature), the Bakry--\'Emery $\Gamma_2$ calculus \cite[Ch.~1, Ch.~3]{BGL14} gives, pointwise in $\Omega$,
\begin{equation}\label{eq:bochner}
\frac12\,\Lo|\nabla h|^2
=\|\nabla^2h\|_F^2+\nabla h\cdot\nabla(\Lo h)+\langle\nabla^2U\,\nabla h,\nabla h\rangle .
\end{equation}
To verify \eqref{eq:bochner}, notice that the classical Bochner formula in $\R^d$,
\begin{align}\label{add:three:1}
\frac12\Delta|\nabla h|^2=\|\nabla^2h\|_F^2+\nabla h\cdot\nabla\Delta h,
\end{align}
and moreover,
\begin{align}\label{add:three:2}
-\frac12\nabla U\cdot\nabla|\nabla h|^2=-\nabla^2h(\nabla h,\nabla U),
\end{align}
while
\begin{align}\label{add:three:3}
\nabla h\cdot\nabla(\nabla U\cdot\nabla h)=\nabla^2h(\nabla h,\nabla U)+\langle\nabla^2U\nabla h,\nabla h\rangle.
\end{align}
By adding \eqref{add:three:1}, \eqref{add:three:2} and \eqref{add:three:3}, we obtain \eqref{eq:bochner}.

\emph{Step 2: integrating the left-hand side.} For any $g\in C^2(\overline\Omega)$, writing $\Lo g\,e^{-U}=\nabla\cdot(e^{-U}\nabla g)$ and applying the divergence theorem with $n$ outward,
\begin{equation}\label{eq:Lo-integral}
\int_\Omega\Lo g\,\rmd\mux=\frac1{Z_x}\int_\Omega\nabla\cdot\left(e^{-U}\nabla g\right)\rmd x=\frac1{Z_x}\int_{\partial\Omega}\partial_ng\,e^{-U}\rmd\sigma .
\end{equation}
Applying \eqref{eq:Lo-integral} with $g=|\nabla h|^2$ gives
\begin{equation}\label{eq:LHS}
\int_\Omega\frac12\Lo|\nabla h|^2\,\rmd\mux=\frac1{2Z_x}\int_{\partial\Omega}\partial_n|\nabla h|^2\,e^{-U}\rmd\sigma .
\end{equation}

\emph{Step 3: integrating the cross term.} By the weighted Green formula \eqref{eq:green} (Lemma~\ref{lem:green}) with $a=\Lo h\in H^1(\Omega)$ and $b=h$, using $\partial_nh=0$ on $\partial\Omega$,
\begin{equation}\label{eq:cross}
\int_\Omega\nabla h\cdot\nabla(\Lo h)\,\rmd\mux
=-\int_\Omega(\Lo h)^2\,\rmd\mux
+\frac1{Z_x}\int_{\partial\Omega}(\Lo h)\underbrace{\partial_nh}_{=0}\,e^{-U}\rmd\sigma
=-\int_\Omega(\Lo h)^2\,\rmd\mux .
\end{equation}

\emph{Step 4: identifying the boundary term.} It remains to compute $\frac12\partial_n|\nabla h|^2$ on $\partial\Omega$. Since $\nabla|\nabla h|^2=2\nabla^2h\,\nabla h$,
\begin{equation}\label{eq:bdry1}
\frac12\partial_n|\nabla h|^2=\big\langle\nabla^2h\,\nabla h,\ n\big\rangle=\nabla^2h(\nabla h,n).
\end{equation}
Now use $\partial_nh=\langle\nabla h,n\rangle\equiv0$ \emph{on $\partial\Omega$}. This is an identity between functions on the hypersurface $\partial\Omega$, so it may be differentiated in any direction $\tau$ tangent to $\partial\Omega$; writing $D_\tau$ for the tangential derivative and using that $n\in C^1(\partial\Omega)$ (here $\partial\Omega\in C^2$ is used),
\[
0=D_\tau\big\langle\nabla h,n\big\rangle=\big\langle\nabla^2h\,\tau,\ n\big\rangle+\big\langle\nabla h,\ D_\tau n\big\rangle
=\nabla^2h(\tau,n)+\II(\tau,\nabla h),
\]
where $\II(\tau,\tau'):=\langle D_\tau n,\tau'\rangle$ is the second fundamental form of $\partial\Omega$ with respect to the \emph{outward} normal, $D_\tau n$ being the shape (Weingarten) operator. Since $\partial_nh=0$, the vector $\nabla h$ is itself tangent to $\partial\Omega$, so we may take $\tau=\nabla h$ and obtain
\begin{equation}\label{eq:bdry2}
\nabla^2h(\nabla h,n)=-\,\II(\nabla h,\nabla h)\qquad\text{on }\partial\Omega.
\end{equation}
Combining \eqref{eq:bdry1} and \eqref{eq:bdry2},
\begin{equation}\label{eq:bdry3}
\frac12\partial_n|\nabla h|^2=-\,\II(\nabla h,\nabla h)\qquad\text{on }\partial\Omega.
\end{equation}
Now integrate \eqref{eq:bochner} against $\mux$ and substitute \eqref{eq:LHS}, \eqref{eq:cross} and \eqref{eq:bdry3}:
\[
-\frac1{Z_x}\int_{\partial\Omega}\II(\nabla h,\nabla h)e^{-U}\rmd\sigma
=\int_\Omega\|\nabla^2h\|_F^2\rmd\mux-\int_\Omega(\Lo h)^2\rmd\mux+\int_\Omega\langle\nabla^2U\nabla h,\nabla h\rangle\rmd\mux,
\]
and rearranging gives \eqref{eq:reilly}.

\emph{Step 5: from $C^3$ to $D(\Lo)$.} Let $h\in D(\Lo)$ and set $g:=(1-\Lo)h\in L^2(\mux)$, so that $h=(1-\Lo)^{-1}g$. Choose $g_k\in C^\infty(\overline\Omega)$ with $g_k\to g$ in $L^2(\mux)$ (possible since $C^\infty(\overline\Omega)$ is dense in $L^2$), and set
\[
h_k:=(1-\Lo)^{-1}g_k\in D(\Lo).
\]
Then:
\begin{enumerate}
\item[(1)] \emph{$h_k$ is sufficiently smooth.} With $\partial\Omega\in C^{2,1}$, the Neumann problem $\Delta w=F\in H^1(\Omega)$, $\partial_nw=0$, has $w\in H^3(\Omega)$ by the standard $H^{k+2}$ regularity theory for $C^{k+1,1}$ boundaries with $k=1$; and $F:=h_k-g_k+\nabla U\cdot\nabla h_k$ is indeed in $H^1(\Omega)$, because $\nabla U\in C^1(\overline\Omega)\subset W^{1,\infty}$ multiplies $\nabla h_k\in H^1$ into $H^1$.
\item[(2)] \emph{Convergence.} $(1-\Lo)^{-1}$ is a bounded operator on $L^2(\mux)$ with the norm less than or equal to $1$, so $h_k\to h$ in $L^2(\mux)$ and $\Lo h_k=h_k-g_k\to h-g=\Lo h$ in $L^2(\mux)$. Consequently $h_k\to h$ in the graph norm of $\Lo$, and since the graph norm dominates the $H^2(\Omega)$ norm on $D(\Lo)$ (see \eqref{eq:apriori-H2} from Lemma~\ref{lem:neumann}) we get $h_k\to h$ in $H^2(\Omega)$.
\item[(3)] \emph{Stability of each term in \eqref{eq:reilly}.} The three interior integrals are continuous with respect to $H^2(\Omega)$ convergence: $\int_{\Omega}(\Lo h)^2\rmd\mux$ and $\int_{\Omega}\|\nabla^2h\|_F^2\rmd\mux$ are quadratic in second derivatives, and $\int_{\Omega}\langle\nabla^2U\nabla h,\nabla h\rangle\rmd\mux$ is quadratic in first derivatives with the bounded coefficient $\nabla^2U\in C^0(\overline\Omega)$. The boundary integral is continuous too: by the trace theorem on the Lipschitz domain $\Omega$,
\begin{align*}
&\Big|\int_{\partial\Omega}\II(\nabla h_k,\nabla h_k)e^{-U}\rmd\sigma-\int_{\partial\Omega}\II(\nabla h,\nabla h)e^{-U}\rmd\sigma\Big|
\\
&\le C\|\II\|_{L^\infty}\|\nabla h_k-\nabla h\|_{L^2(\partial\Omega)}\left(\|\nabla h_k\|_{L^2(\partial\Omega)}+\|\nabla h\|_{L^2(\partial\Omega)}\right),
\end{align*}
and
\[
\|\nabla h_k-\nabla h\|_{L^2(\partial\Omega)}\le C_{\mathrm{tr}}\|\nabla h_k-\nabla h\|_{H^1(\Omega)}\le C_{\mathrm{tr}}\|h_k-h\|_{H^2(\Omega)}\to0,
\]
where $\|\II\|_{L^\infty}<\infty$ because $\partial\Omega\in C^2$ is compact.
\item[(4)] \emph{The Neumann trace is preserved along the approximation.} That is, every $h_k$ lies in $D(\Lo)$, hence satisfies $\partial_nh_k=0$ exactly --- not approximately --- so the term $\int_{\partial\Omega}(\Lo h_k)\partial_nh_k\,e^{-U}\rmd\sigma$ in \eqref{eq:cross} vanishes for every $k$, and \eqref{eq:bdry3} holds for every $k$.
\end{enumerate}
Passing to the limit in \eqref{eq:reilly} written for $h_k$ gives \eqref{eq:reilly} for $h$.

\emph{Step 6: the inequality \eqref{eq:hessian-bound}.} Rearranging \eqref{eq:reilly},
\begin{align}\label{last:middle}
\int_\Omega\|\nabla^2h\|_F^2\rmd\mux
=\int_\Omega(\Lo h)^2\rmd\mux
-\int_\Omega\langle\nabla^2U\nabla h,\nabla h\rangle\rmd\mux
-\frac1{Z_x}\int_{\partial\Omega}\II(\nabla h,\nabla h)e^{-U}\rmd\sigma .
\end{align}
The last term in \eqref{last:middle} satisfies
\[
-\frac1{Z_x}\int_{\partial\Omega}\II(\nabla h,\nabla h)e^{-U}\rmd\sigma\leq 0,
\]
because $\Omega$ is convex, hence $\II\succeq0$ pointwise on $\partial\Omega$ (with the outward normal convention: for the ball of radius $R$, $D_\tau n=\tau/R$ and $\II(\tau,\tau)=|\tau|^2/R>0$). The middle term in \eqref{last:middle} satisfies
\[
-\int_\Omega\langle\nabla^2U\nabla h,\nabla h\rangle\rmd\mux\le K\int_{\Omega}|\nabla h|^2\rmd\mux,
\]
because $\nabla^2U\succeq-K\,\mathrm{Id}$. This gives \eqref{eq:hessian-bound}.
\end{proof}

We can now state the corrector bounds; they are identical to \cite[Lemma~1]{FLL26}, and we indicate in the proof exactly where the boundary enters.
\begin{lemma}[Bounds on the corrector]\label{lem:corrector}
For any $\phi\in\mathscr C_0$,
\begin{align}
\nrm{\Am\phi}&\le\frac1{2\sqrt m}\,\nrm{\phi},\label{eq:Am-bound}\\
\nrm{\La\Am\phi}&\le\nrm{\phi},\label{eq:LaAm-bound}\\
\nrm{\Am\La(1-\Pv)\phi}&\le\sqrt{2+\frac{K}{2m}}\;\nrm{\phi}.\label{eq:AmLa-bound}
\end{align}
In particular, $\Am$, $\La\Am$ and $\Am\La(1-\Pv)$ extend uniquely to bounded operators on $\Lmuz$ with the same bounds, and if $\varepsilon<\sqrt m$ then
\begin{equation}\label{eq:equivalence}
\frac{1-\varepsilon/\sqrt m}{2}\,\nrm{\phi}^2\le\mathsf L_m(\phi)\le\frac{1+\varepsilon/\sqrt m}{2}\,\nrm{\phi}^2
\qquad\text{for any $\phi\in\Lmuz$}.
\end{equation}
\end{lemma}
\begin{proof}
Set $u=\Am\phi$ and $w=(\La\Pv)^*\phi=-\Pv\La\phi$, so that $u=(m-\Lo)^{-1}w$, and both are functions of $x$ only.

\emph{Step 0: $w$ and $u$ have mean zero.} By Lemma~\ref{lem:antisym} applied to the pair $(\phi,1)$ --- both are in $\mathscr C$, the constant function $1$ being trivially specular and satisfying \eqref{eq:poly-growth} --- and using $\La1=0$,
\[
\int_\Omega w\,\rmd\mux=-\int_\Omega \Pv\La\phi\,\rmd\mux=-\int_{\Omega\times\R^d}\La\phi\,\rmd\mu=-\ip{\La\phi}{1}=\ip{\phi}{\La1}=0 ,
\]
where the second equality uses the fact that $\int_\Omega(\Pv F)\rmd\mux=\int_{\Omega\times\mathbb{R}^{d}} F\rmd\mu$ for any $F$, by the product structure $\mu=\mux\otimes\kappa$. Hence $w\in L^2_0(\mux)$, and since $(m-\Lo)^{-1}$ maps $L^2_0(\mux)$ into itself (it commutes with the spectral projection onto $\mathrm{span}\{1\}^\perp$), also $u\in L^2_0(\mux)$. Moreover $u\in D(\Lo)$, so $\partial_nu=0$ and $\La u=v\cdot\nabla_xu$ is specular by Remark~\ref{rem:neumann-specular}.

\emph{Step 1: proof of \eqref{eq:Am-bound}--\eqref{eq:LaAm-bound}.} Applying \eqref{eq:form-identity} with $g=h=u$ gives
\[
\nrm{\La u}^2=\mathcal E(u,u)=\ip[\mux]{u}{-\Lo u}.
\]
Therefore, using $w=(m-\Lo)u$,
\begin{equation}\label{eq:key-identity}
\begin{aligned}
m\nrm[\mux]{u}^2+\nrm{\La u}^2
&=\ip[\mux]{mu-\Lo u}{u}
\\
&=\ip[\mux]{w}{u}
=\ip{\phi}{\La u}
\le\nrm{\phi}\,\nrm{\La u},
\end{aligned}
\end{equation}
where the third equality is
\[
\ip[\mux]{w}{u}=\ip{(\La\Pv)^*\phi}{u}=\ip{\phi}{\La\Pv u}=\ip{\phi}{\La u},
\]
by the definition of the adjoint together with $\Pv u=u$ ($u$ being $x$-only); this is where Lemma~\ref{lem:antisym-sob} is used, through \eqref{eq:adjoint-LaPv}, and there is no boundary term because both $\phi$ and $\La u$ are specular.\footnote{Lemma~\ref{lem:antisym-sob} is applied here with $f=\phi$ and $g=u$, i.e.\ $g_0=u\in D(\Lo)$ and $G=0$. Then \eqref{eq:sob-specular} is vacuous --- a $v$-independent function is specular automatically --- and only $u\in H^1(\Omega)$ is needed. This covers the use of \eqref{eq:adjoint-LaPv} in Step~1 (the identity $\ip[\mux]{w}{u}=\ip{\phi}{\La u}$).} Dropping the nonnegative term $m\nrm[\mux]{u}^2$ from the left of \eqref{eq:key-identity} and canceling one factor $\nrm{\La u}$ gives \eqref{eq:LaAm-bound}. For \eqref{eq:Am-bound}, apply the elementary inequality $a^2+b^2\ge2ab$ with $a=\sqrt m\,\nrm[\mux]{u}$ and $b=\nrm{\La u}$ to the left-hand side of \eqref{eq:key-identity}:
\[
2\sqrt m\,\nrm[\mux]{u}\,\nrm{\La u}\le m\nrm[\mux]{u}^2+\nrm{\La u}^2\le\nrm{\phi}\,\nrm{\La u}.
\]
If $\nrm{\La u}>0$ we may cancel it and obtain $\nrm[\mux]{u}\le\nrm{\phi}/(2\sqrt m)$. If $\nrm{\La u}=0$, then by \eqref{eq:form-identity} $\nabla_xu=0$, so $u$ is constant on the connected set $\Omega$, and being mean zero, $u=0$; the bound is then trivial. Finally $\nrm{u}=\nrm[\mux]{u}$ because $u$ does not depend on $v$ and $\kappa$ is a probability measure. This is \eqref{eq:Am-bound}.

\emph{Step 2: proof of \eqref{eq:AmLa-bound}; the adjoint $B^*$.} Write $B:=\Am\La(1-\Pv)$, defined on $\mathscr C_0$. By \eqref{eq:Am-a} and Step~0, the range of $B$ is contained in $\Ran\Pv\cap\Lmuz$, so for $\phi\in\mathscr C_0$,
\[
\nrm{B\phi}=\sup\big\{\ip{B\phi}{g}\ :\ g\in\Ran\Pv\cap\Lmuz,\ \nrm{g}=1\big\},
\]
and for every such $g$ we justify below the duality identity $\ip{B\phi}{g}=\ip{\phi}{B^*g}$ with an explicitly computed $B^*g$; consequently,\footnote{We do not presume $B$ bounded --- the bound \eqref{eq:normB} is what proves it.}
\begin{equation}\label{eq:normB}
\nrm{B\phi}\le\left(\sup\big\{\nrm{B^*g}\ :\ g\in\Ran\Pv\cap\Lmuz,\ \nrm{g}=1\big\}\right)\,\nrm{\phi}.
\end{equation}
We now compute $B^*$. Using $(ST)^*=T^*S^*$, the self-adjointness of $\Pv$ and of $(m-\Lo)^{-1}$, the antisymmetry $\La^*=-\La$ from Lemma~\ref{lem:antisym-sob} (Indeed, the step $\La^*=-\La$ inside \eqref{eq:Bstar} is used against the pair $\left((1-\Pv)\psi,\ \La h\right)$ with $\psi\in\mathscr C_0$ and $h=(m-\Lo)^{-1}g\in D(\Lo)$. Now $(1-\Pv)\psi\in\mathscr C$, and $\La h=v\cdot\nabla_xh$ is of the form $g_0+v\cdot G$ with $g_0=0$ and $G=\nabla_xh\in H^1(\Omega;\R^d)$ satisfying $n\cdot G=\partial_nh=0$ $\sigma$-a.e.; so Lemma~\ref{lem:antisym-sob} applies and gives \eqref{eq:antisym-sob} with no boundary term.), and \eqref{eq:adjoint-LaPv},
\begin{equation}\label{eq:Bstar}
\begin{aligned}
B^*&=\left(\Am\La(1-\Pv)\right)^*
=(1-\Pv)^*\,\La^*\,\Am^*
=-(1-\Pv)\,\La\,\Am^*,
\\
\Am^*&=\left((m-\Lo)^{-1}(\La\Pv)^*\right)^*=(\La\Pv)\,\left((m-\Lo)^{-1}\right)^*=\La\Pv\,(m-\Lo)^{-1},
\end{aligned}
\end{equation}
and hence
\begin{equation}\label{eq:Bstar2}
B^*=-(1-\Pv)\,\La\,\La\,\Pv\,(m-\Lo)^{-1}=-(1-\Pv)\La^2\Pv(m-\Lo)^{-1}.
\end{equation}

\emph{Step 3: computing $B^*g$ explicitly.} Fix $g\in\Ran\Pv\cap\Lmuz$ with $\nrm{g}=1$, and set
\begin{equation}\label{eq:h-def}
h:=(m-\Lo)^{-1}g\in D(\Lo)\cap L^2_0(\mux),
\end{equation}
a function of $x$ alone with $\partial_nh=0$. Then $\Pv h=h$ and, by \eqref{eq:LaPv}, $\La h=v\cdot\nabla_xh$. Applying $\La$ once more, and using $\nabla_v(v\cdot\nabla_xh)=\nabla_xh$,
\[
\La^2h=\La\left(v\cdot\nabla_xh\right)
=v\cdot\nabla_x\left(v\cdot\nabla_xh\right)-\nabla_xU\cdot\nabla_v\left(v\cdot\nabla_xh\right)
=\sum_{i,j}v_iv_j\,\partial_{ij}h-\nabla_xU\cdot\nabla_xh .
\]
Averaging in $v$ with $\int_{\mathbb{R}^{d}} v_iv_j\rmd\kappa=\delta_{ij}$,
\[
\Pv\La^2h=\sum_{i,j}\delta_{ij}\partial_{ij}h-\nabla_xU\cdot\nabla_xh=\Delta_xh-\nabla_xU\cdot\nabla_xh=\Lo h,
\]
and therefore,
\begin{equation}\label{eq:Bstar-explicit}
B^*g=-(1-\Pv)\La^2h=-\left(\sum_{i,j}v_iv_j\partial_{ij}h-\nabla_xU\cdot\nabla_xh\right)+\Lo h
=-\sum_{i,j}\left(v_iv_j-\delta_{ij}\right)\partial_{ij}h .
\end{equation}
Note that the first-order term $\nabla_xU\cdot\nabla_xh$, which is $v$-independent, is annihilated by $1-\Pv$; only the traceless part of the quadratic-in-$v$ term survives.

\emph{Step 4: the norm of $B^*g$.} By \eqref{eq:Bstar-explicit} and the product structure of $\mu$,
\[
\nrm{B^*g}^2=\int_\Omega\sum_{i,j,k,l}\partial_{ij}h\,\partial_{kl}h\left(\int_{\R^d}(v_iv_j-\delta_{ij})(v_kv_l-\delta_{kl})\rmd\kappa(v)\right)\rmd\mux(x).
\]
By the Gaussian fourth-moment and second-moment identities $\int_{\mathbb{R}^{d}} v_iv_jv_kv_l\rmd\kappa=\delta_{ij}\delta_{kl}+\delta_{ik}\delta_{jl}+\delta_{il}\delta_{jk}$ and $\int_{\mathbb{R}^{d}} v_iv_j\rmd\kappa=\delta_{ij}$,
\[
\int_{\R^d}(v_iv_j-\delta_{ij})(v_kv_l-\delta_{kl})\rmd\kappa
=\left(\delta_{ij}\delta_{kl}+\delta_{ik}\delta_{jl}+\delta_{il}\delta_{jk}\right)-\delta_{ij}\delta_{kl}-\delta_{kl}\delta_{ij}+\delta_{ij}\delta_{kl}
=\delta_{ik}\delta_{jl}+\delta_{il}\delta_{jk}.
\]
Contracting,
\[
\sum_{i,j,k,l}\left(\delta_{ik}\delta_{jl}+\delta_{il}\delta_{jk}\right)\partial_{ij}h\,\partial_{kl}h
=\sum_{i,j}(\partial_{ij}h)^2+\sum_{i,j}\partial_{ij}h\,\partial_{ji}h=2\|\nabla^2_xh\|_F^2,
\]
using the symmetry $\partial_{ij}h=\partial_{ji}h$. Hence,
\begin{equation}\label{eq:Bstar-norm}
\nrm{B^*g}^2=2\,\nrm[\mux]{\nabla^2_xh}^2 .
\end{equation}

\emph{Step 5: the Reilly estimate and spectral calculus.} By \eqref{eq:hessian-bound} and the identity $\nrm[\mux]{\nabla_xh}^2=\ip[\mux]{h}{-\Lo h}$ (which is \eqref{eq:green} with $a=b=h$ and $\partial_nh=0$),
\[
\nrm{B^*g}^2\le2\nrm[\mux]{\Lo h}^2+2K\,\ip[\mux]{h}{-\Lo h}.
\]
Let $(E_\lambda)$ be the spectral resolution of the nonnegative self-adjoint operator $-\Lo$ restricted to $L^2_0(\mux)$, so that $\Spec(-\Lo|_{L^2_0(\mux)})\subset[m,\infty)$ by \eqref{eq:gap}, and write $\rmd\nu_g(\lambda):=\rmd\ip[\mux]{E_\lambda g}{g}$, a positive measure on $[m,\infty)$ of total mass $\nrm[\mux]{g}^2$. Since $h=(m-\Lo)^{-1}g$ corresponds to the multiplier $(m+\lambda)^{-1}$,
\[
\nrm[\mux]{\Lo h}^2=\int_m^\infty\frac{\lambda^2}{(m+\lambda)^2}\rmd\nu_g(\lambda),
\qquad
\ip[\mux]{h}{-\Lo h}=\int_m^\infty\frac{\lambda}{(m+\lambda)^2}\rmd\nu_g(\lambda),
\]
so that
\[
\nrm{B^*g}^2\le\int_m^\infty\frac{2\lambda^2+2K\lambda}{(m+\lambda)^2}\rmd\nu_g(\lambda)
\le\left(\sup_{\lambda\ge m}\frac{2\lambda^2+2K\lambda}{(m+\lambda)^2}\right)\nrm[\mux]{g}^2 .
\]
Bounding the two terms of the numerator separately:
\[
\sup_{\lambda\ge m}\frac{\lambda^2}{(m+\lambda)^2}=\lim_{\lambda\to\infty}\frac{\lambda^2}{(m+\lambda)^2}=1,
\qquad
\sup_{\lambda\ge m}\frac{\lambda}{(m+\lambda)^2}=\frac{m}{(2m)^2}=\frac1{4m},
\]
where the second equation above is due to $\frac{\rmd}{\rmd\lambda}\frac{\lambda}{(m+\lambda)^2}=\frac{m-\lambda}{(m+\lambda)^3}\le0$ for $\lambda\ge m$, so the supremum is attained at $\lambda=m$. Therefore,
\[
\nrm{B^*g}^2\le\left(2+\frac{2K}{4m}\right)\nrm[\mux]{g}^2=\left(2+\frac{K}{2m}\right)\nrm[\mux]{g}^2 ,
\]
which with \eqref{eq:normB} proves \eqref{eq:AmLa-bound}.

\emph{Step 6: the norm equivalence \eqref{eq:equivalence}.} By Cauchy--Schwarz inequality and \eqref{eq:Am-bound},
\[
\big|\ip{\Am\phi}{\phi}\big|\le\nrm{\Am\phi}\,\nrm{\phi}\le\frac1{2\sqrt m}\nrm{\phi}^2,
\]
so from the definition \eqref{eq:Am-def} of $\mathsf L_m$,
\[
\frac12\nrm{\phi}^2-\frac{\varepsilon}{2\sqrt m}\nrm{\phi}^2\le\mathsf L_m(\phi)\le\frac12\nrm{\phi}^2+\frac{\varepsilon}{2\sqrt m}\nrm{\phi}^2,
\]
which is \eqref{eq:equivalence}. All bounds extend from $\mathscr C_0$ to $\Lmuz$ by density (Assumption~\ref{ass:WP}).
\end{proof}
\section{Proof of Theorem~\ref{thm:main}}\label{sec:proof}
Let $f(t)=e^{t\Lgen}f_0$ with $f_0\in\mathscr C_0$; the extension to general $f_0\in\Lmuz$ is performed at the end.

\emph{Step 1: differentiating the modified functional.} By Assumption~\ref{ass:WP} the map $t\mapsto f(t)$ is $C^1$ with values in $\Lmuz$ and takes values in $D(\Lgen)$, so we may differentiate \eqref{eq:Am-def} term by term. Using $\partial_tf=\Lgen f$ and the boundedness of $\Am$ (Lemma~\ref{lem:corrector}),
\begin{equation}\label{eq:dLm}
\frac{\rmd}{\rmd t}\mathsf L_m(f(t))
=\ip{f}{\Lgen f}-\varepsilon\left(\ip{\Am\Lgen f}{f}+\ip{\Am f}{\Lgen f}\right)
=-\mathscr D_\varepsilon(f(t)),
\end{equation}
with
\begin{equation}\label{eq:Deps}
\mathscr D_\varepsilon(f):=-\ip{\Lgen f}{f}
+\varepsilon\left(\ip{\Am\Lgen f}{f}+\ip{\Am f}{\Lgen f}\right).
\end{equation}
Fix $f\in\mathscr C_0$ and split it as
\[
f_S:=\Pv f,
\qquad
f_F:=(1-\Pv)f,
\]
the ``slow'' (macroscopic, $v$-independent) and ``fast'' (microscopic) components; note $f_S\in L^2_0(\mux)$ since $f\in\Lmuz$, and $\nrm{f}^2=\nrm{f_S}^2+\nrm{f_F}^2$ by orthogonality.

\emph{Step 2: the uncorrected dissipation.} By \eqref{eq:energy},
\begin{equation}\label{eq:uncorrected}
-\ip{\Lgen f}{f}=\gamma\nrm{\nabla_vf}^2 .
\end{equation}
This is where the antisymmetry of $\La$ (Lemma~\ref{lem:antisym}), hence the specular boundary condition, is used for the first time. Since $\nrm{\nabla_vf}$ vanishes on $\ker\Ls$, \eqref{eq:uncorrected} alone gives no decay, and the $\varepsilon$-correction in \eqref{eq:Deps} must supply coercivity on $f_S$. We split that correction according to $\Lgen=\La+\gamma\Ls$:
\begin{align}\label{eq:split}
\ip{\Am\Lgen f}{f}+\ip{\Am f}{\Lgen f}
&=\underbrace{\gamma\left(\ip{\Am\Ls f}{f}+\ip{\Am f}{\Ls f}\right)}_{\Ls\text{-contribution}}
\\
&\qquad+\underbrace{\ip{\Am\La f}{f}+\ip{\Am f}{\La f}}_{\La\text{-contribution}}.
\nonumber
\end{align}

\emph{Step 3: the $\Ls$ contribution.} First, since $\Am=\Pv\Am$ by \eqref{eq:Am-a} and $\Pv\Ls=0$ (because $\int_{\R^d}\Ls F\,\rmd\kappa=\ip[\kappa]{\Ls F}{1}=\ip[\kappa]{F}{\Ls1}=0$),
\begin{equation}\label{eq:Ls-first}
\ip{\Am f}{\Ls f}=\ip{\Pv\Am f}{\Ls f}=\ip{\Am f}{\Pv\Ls f}=0 .
\end{equation}
For the remaining term, note that $\Ls f=\Ls f_F$ (since $f_S$ is $v$-independent, $\Ls f_S=0$) and that $\Am\Ls f$ is a function of $x$ alone, so pairing it with $f$ is the same as pairing it with $\Pv f=f_S$:
\[
\ip{\Am\Ls f}{f}=\ip{\Am\Ls f_F}{\Pv f}=\ip{\Am\Ls f_F}{f_S}.
\]
Now move $\Am$ to the other side. By \eqref{eq:Bstar} we have $\Am^*=\La\Pv(m-\Lo)^{-1}$; setting
\begin{equation}\label{eq:phi-def}
\phi:=(m-\Lo)^{-1}f_S\in D(\Lo)\cap L^2_0(\mux),\qquad \partial_n\phi=0,
\end{equation}
we get
\[
\Am^*f_S=\La\Pv\phi=\La\phi=v\cdot\nabla_x\phi,
\]
which is specular by Remark~\ref{rem:neumann-specular}. Hence, using \eqref{eq:Ls-sym},\footnote{\eqref{eq:Ls-sym} is applied here with $g=\La\phi$, which is merely $H^1$ in $x$; this is legitimate because $\Ls$ acts in the $v$-variable only: the fixed-$x$ identity \eqref{by:integrating} holds for a.e.\ $x$ (with $f_F(x,\cdot)\in C^2$ of polynomial growth and $\La\phi(x,\cdot)$ affine in $v$), and the subsequent integration in $x$ against $\mux$ is justified by the Cauchy--Schwarz inequality.}
\begin{equation}\label{eq:Ls-chain}
\begin{aligned}
\ip{\Am\Ls f_F}{f_S}
&=\ip{\Ls f_F}{\Am^*f_S}
=\ip{\Ls f_F}{\La\phi}
\\
&=-\ip{\nabla_vf_F}{\nabla_v(\La\phi)}
=-\ip{\nabla_vf}{\nabla_x\phi},
\end{aligned}
\end{equation}
where in the last equality we used the two elementary identities
\[
\nabla_v(\La\phi)=\nabla_v\left(v\cdot\nabla_x\phi\right)=\nabla_x\phi,
\qquad
\nabla_vf_F=\nabla_v\left((1-\Pv)f\right)=\nabla_vf-\nabla_v(\Pv f)=\nabla_vf .
\]
It remains to bound $\nrm{\nabla_x\phi}$. Since $\phi$ is $x$-only, $\nrm{\nabla_x\phi}=\nrm[\mux]{\nabla_x\phi}$, and by \eqref{eq:green} with $a=b=\phi$ and $\partial_n\phi=0$,
\[
\nrm[\mux]{\nabla_x\phi}^2=\ip[\mux]{\phi}{-\Lo\phi}=\big\|(-\Lo)^{1/2}(m-\Lo)^{-1}f_S\big\|^2_{L^2(\mux)} .
\]
By the spectral theorem for $-\Lo$ on $L^2_0(\mux)$ and \eqref{eq:gap}, the multiplier is $\sqrt\lambda/(m+\lambda)$ with $\lambda\ge m$, so
\[
\nrm[\mux]{\nabla_x\phi}\le\left(\sup_{\lambda\ge m}\frac{\sqrt\lambda}{m+\lambda}\right)\nrm[\mux]{f_S}=\frac{1}{2\sqrt m}\nrm[\mux]{f_S},
\]
the supremum being attained at $\lambda=m$ because $\frac{\rmd}{\rmd\lambda}\frac{\sqrt\lambda}{m+\lambda}=\frac{m-\lambda}{2\sqrt\lambda(m+\lambda)^2}\le0$ for $\lambda\ge m$, with value $\sqrt m/(2m)=1/(2\sqrt m)$. Combining with \eqref{eq:Ls-chain} and Cauchy--Schwarz inequality,
\begin{equation}\label{eq:Ls-term}
\left|\ip{\Am\Ls f}{f}\right|\le\frac1{2\sqrt m}\,\nrm{\nabla_vf}\,\nrm{f_S}.
\end{equation}

\emph{Step 4: the $\La$ contribution.} Using $\Am=\Pv\Am$, $\Am\Pv=0$ and $\Pv\La f=\Pv\La f_F$ (the last from $\Pv\La\Pv=0$ in \eqref{eq:LaPv}), we decompose
\begin{equation}\label{eq:La-decomp}
\begin{aligned}
&\ip{\Am\La f}{f}+\ip{\Am f}{\La f}
\\
&=\ip{\Am\La f}{\Pv f}+\ip{\Am(1-\Pv)f}{\Pv\La f}
\\
&=\ip{\Am\La f_S}{f_S}+\ip{\Am\La f_F}{f_S}+\ip{\Am f_F}{\La f_F} .
\end{aligned}
\end{equation}
(In the first line: $\Am\La f$ is $x$-only, so the first pairing sees only $\Pv f$; and $\Am f=\Am(1-\Pv)f$ by \eqref{eq:Am-b}, while $\Am f_F$ is $x$-only so the second pairing sees only $\Pv\La f$. In the second line we expanded $\La f=\La f_S+\La f_F$ in the first pairing and used $\Pv\La f=\Pv\La f_F$ together with $\Pv(\Am f_F)=\Am f_F$ in the second.)
The first term of \eqref{eq:La-decomp} is the coercive one. By \eqref{eq:Am-c},
\[
\Am\La f_S=\Am\La\Pv f=(m-\Lo)^{-1}(-\Lo)f_S,
\]
and by the spectral theorem on $L^2_0(\mux)$ with \eqref{eq:gap},
\begin{equation}\label{eq:coercive-term}
\begin{aligned}
\ip{\Am\La f_S}{f_S}
&=\left\langle(m-\Lo)^{-1}(-\Lo)f_S,\,f_S\right\rangle_{L^2(\mux)}
\\
&=\int_m^\infty\frac{\lambda}{m+\lambda}\rmd\nu_{f_S}(\lambda)
\\
&\ge\left(\min_{\lambda\ge m}\frac{\lambda}{m+\lambda}\right)\,\nrm[\mux]{f_S}^2
=\frac12\,\nrm[\mux]{f_S}^2,
\end{aligned}
\end{equation}
the minimum being attained at $\lambda=m$ since $\lambda\mapsto\lambda/(m+\lambda)$ is increasing. 
The two remaining terms of \eqref{eq:La-decomp} are error terms, controlled by Lemma~\ref{lem:corrector}. Since $(1-\Pv)f_F=f_F$, \eqref{eq:AmLa-bound} applied to $\phi=f_F$ gives
\begin{equation}\label{eq:err1}
\left|\ip{\Am\La f_F}{f_S}\right|\le\nrm{\Am\La(1-\Pv)f_F}\,\nrm{f_S}\le\sqrt{2+\frac K{2m}}\,\nrm{f_F}\,\nrm{f_S},
\end{equation}
and by Lemma~\ref{lem:antisym-sob} followed by \eqref{eq:LaAm-bound},
\footnote{Lemma~\ref{lem:antisym-sob} is applied here with $f=f_F$ and $g=u:=\Am f_F\in D(\Lo)$, i.e.\ $g_0=u$ and $G=0$; then \eqref{eq:sob-specular} is vacuous --- a $v$-independent function is specular automatically --- and only $u\in H^1(\Omega)$ is needed.}
\begin{equation}\label{eq:err2}
\left|\ip{\Am f_F}{\La f_F}\right|=\left|\ip{\La\Am f_F}{f_F}\right|\le\nrm{\La\Am f_F}\,\nrm{f_F}\le\nrm{f_F}^2.
\end{equation}
Combining \eqref{eq:La-decomp}--\eqref{eq:err2},
\begin{equation}\label{eq:La-total}
\begin{aligned}
&\ip{\Am\La f}{f}+\ip{\Am f}{\La f}
\\
&\ge \frac12\nrm{f_S}^2-\sqrt{2+\frac K{2m}}\,\nrm{f_F}\nrm{f_S}-\nrm{f_F}^2 .
\end{aligned}
\end{equation}

\emph{Step 5: the quadratic form.} Insert \eqref{eq:uncorrected}, \eqref{eq:Ls-first}, \eqref{eq:Ls-term} and \eqref{eq:La-total} into \eqref{eq:Deps}--\eqref{eq:split}, and use $\nrm{f_F}\le\nrm{\nabla_vf}$, which is the velocity Poincar\'e inequality \eqref{eq:gauss-poincare} integrated in $x$:
\begin{align*}
\mathscr D_\varepsilon(f)
&\ge\gamma\nrm{\nabla_vf}^2
\\
&\qquad
+\varepsilon\left(\frac12\nrm{f_S}^2
-\left(\frac{\gamma}{2\sqrt m}+\sqrt{2+\frac K{2m}}\right)\nrm{\nabla_vf}\nrm{f_S}
-\nrm{\nabla_vf}^2\right).
\end{align*}
Setting
\begin{equation}\label{eq:abz}
\alpha:=\nrm{\nabla_vf},\qquad\beta:=\nrm{f_S},\qquad
\zeta:=\frac{\gamma}{2\sqrt m}+\sqrt{2+\frac K{2m}},
\end{equation}
this reads exactly as the matrix bound of \cite{FLL26}:
\begin{equation}\label{eq:matrix}
\mathscr D_\varepsilon(f)\ge
\begin{pmatrix}\alpha&\beta\end{pmatrix}
M_{\gamma,\varepsilon}
\begin{pmatrix}\alpha\\\beta\end{pmatrix},
\qquad
M_{\gamma,\varepsilon}=
\begin{pmatrix}
\gamma-\varepsilon&-\frac{\varepsilon\zeta}{2}\\[2pt]
-\frac{\varepsilon\zeta}{2}&\frac{\varepsilon}{2}
\end{pmatrix}.
\end{equation}

\emph{Step 6: optimization and conclusion.} Since $\zeta$, and hence $M_{\gamma,\varepsilon}$, is \emph{identical} to the whole-space case, the optimization of $(\gamma,\varepsilon)$ in \cite[\S3]{FLL26} applies verbatim; the outcome is that the choices
\[
\gamma_*=\sqrt{16m+2K},\qquad \varepsilon_*=\frac{\gamma_*}{a(\gamma_*)},\qquad a(\gamma):=2+\zeta^2,
\]
give
\begin{equation}\label{eq:lam-coer}
\lambda_{\min}(M_{\gamma_*,\varepsilon_*})\ \ge\
\frac{\sqrt m}{4\left(\sqrt{2+\frac K{2m}}+\sqrt{4+\frac K{2m}}\right)}\ =:\ \lambda_{\mathrm{coer}},
\end{equation}
with
\begin{equation}
\varepsilon_*=\frac{\sqrt m}{\sqrt{4+\frac{K}{2m}}+\sqrt{2+\frac{K}{2m}}}\le\frac{\sqrt m}{2}.
\end{equation}
Since $\varepsilon_*\le\sqrt m/2<\sqrt m$, the norm equivalence \eqref{eq:equivalence} gives
\begin{equation}\label{eq:equiv-numbers}
\frac14\nrm f^2\le\mathsf L_m(f)\le\frac34\nrm f^2 .
\end{equation}
Moreover, by \eqref{eq:gauss-poincare},
\[
\nrm f^2=\nrm{f_F}^2+\nrm{f_S}^2\le\alpha^2+\beta^2,
\]
so \eqref{eq:matrix} and the upper bound in \eqref{eq:equiv-numbers} give
\begin{equation}\label{eq:final-coercivity}
\mathscr D_{\varepsilon_*}(f)\ \ge\ \lambda_{\mathrm{coer}}\left(\alpha^2+\beta^2\right)\ \ge\ \lambda_{\mathrm{coer}}\nrm f^2\ \ge\ \frac{4\lambda_{\mathrm{coer}}}{3}\,\mathsf L_m(f)
\qquad\text{for every }f\in\mathscr C_0 .
\end{equation}
As $\Am$ is bounded on $\Lmuz$ by \eqref{eq:Am-bound}, the form $\mathscr D_{\varepsilon_*}$ is continuous with respect to the graph norm of $\Lgen$; since $\mathscr C_0$ is a core (Assumption~\ref{ass:WP}), \eqref{eq:final-coercivity} extends to all $f\in D(\Lgen)$. Combining \eqref{eq:dLm} with \eqref{eq:final-coercivity} and setting
\[
\Lambda:=\frac23\lambda_{\mathrm{coer}}=\frac16\,\frac{\sqrt m}{\sqrt{2+\frac K{2m}}+\sqrt{4+\frac K{2m}}},
\]
we obtain
\[
\frac{\rmd}{\rmd t}\mathsf L_m(f(t))\le-2\Lambda\,\mathsf L_m(f(t)),
\]
and hence by Gr\"onwall's lemma
\[
\mathsf L_m(f(t))\le e^{-2\Lambda t}\mathsf L_m(f_0).
\]
Applying \eqref{eq:equiv-numbers} at both ends,
\[
\frac14\nrm{f(t)}^2\le\mathsf L_m(f(t))\le e^{-2\Lambda t}\mathsf L_m(f_0)\le\frac34e^{-2\Lambda t}\nrm{f_0}^2,
\]
that is,
\[
\nrm{f(t)}^2\le3e^{-2\Lambda t}\nrm{f_0}^2,
\]
which is \eqref{eq:main-decay} after taking square roots. The estimate extends to arbitrary $f_0\in\Lmuz$ by density and strong continuity of the semigroup: if $f_0^{(k)}\in\mathscr C_0$ with $f_0^{(k)}\to f_0$ in $\Lmuz$, then $e^{t\Lgen}f_0^{(k)}\to e^{t\Lgen}f_0$ for each $t$ by contractivity, and the inequality passes to the limit. Finally, when $K=0$ we have $\gamma_*=4\sqrt m$ and
\[
\Lambda=\frac16\cdot\frac{\sqrt m}{\sqrt2+2}=\frac{\sqrt m}{6(2+\sqrt2)}=\frac{(2-\sqrt2)\sqrt m}{6(2+\sqrt2)(2-\sqrt2)}=\frac{(2-\sqrt2)\sqrt m}{6\cdot2}=\frac{2-\sqrt2}{12}\sqrt m .
\]
This completes the proof.
\qed


\section{Extension to non-convex domains}\label{sec:nonconvex}

In this section, we extend the preceding results to bounded non-convex
domains. The only place in the proof of Theorem~\ref{thm:main} where
convexity is essential is the sign of the boundary term in the weighted
Reilly formula \eqref{eq:reilly}. On a non-convex domain, by applying a trace inequality \eqref{eq:trace-nc} to the second order estimate \eqref{eq:pretrace-nc} produces a term \(\theta\|\nabla_x^2h\|^2\) on the right-hand side, which is absorbed into the left-hand side to obtain \eqref{eq:hessian-nc}, hence, extends the key inequality \eqref{eq:hessian-bound} to a semiconvex domain.
As a result, the constant $2+\frac{K}{2m}$ of
Lemma~\ref{lem:corrector} is replaced by an explicit constant
$A=A(m,K,\sigma_0,\Theta)$, see \eqref{eq:A-def-nc}, in which the negative
part $\sigma_0$ of the boundary curvature enters in the same way as the
negative part $K$ of the potential curvature; everything else in the
hypocoercivity argument is untouched.

\subsection{The semiconvexity parameter}\label{subsec:nc-setting}
Throughout this section, $\Omega\subset\R^d$ is a bounded \emph{connected}
domain with $\partial\Omega\in C^{2,1}$, not necessarily convex, and
$U\in C^2(\overline\Omega)$. The measures $\mu,\mux,\kappa$, the operators
$\Lgen,\La,\Ls,\Pv,\Lo,\Am$, the specular class $\mathscr C$, and the
constant
$Z_x=\int_\Omega e^{-U(x)}\,\rmd x$
are defined exactly as in Section~\ref{sec:setting}. As in
Lemma~\ref{lem:reilly},
\[
\II(\tau,\tau'):=\langle D_\tau n,\tau'\rangle
\]
denotes the second fundamental form of $\partial\Omega$ with respect to
the outward unit normal.

\begin{assumption}[Semiconvex boundary]\label{ass:semiconvex}
There exists $\sigma_0\ge0$ such that
\begin{equation}\label{eq:semiconvex}
\II(\tau,\tau)\ge -\sigma_0|\tau|^2
\qquad
\text{for every }x\in\partial\Omega
\text{ and every }\tau\in T_x\partial\Omega .
\end{equation}
\end{assumption}

When $\Omega$ is convex, Assumption~\ref{ass:semiconvex} holds with
$\sigma_0=0$, so the results below contain those of
Section~\ref{sec:main} as a special case. More generally,
Assumption~\ref{ass:semiconvex} imposes no additional geometric
restriction on a bounded $C^{2,1}$ domain: since $\II$ is continuous on
the compact unit tangent bundle of $\partial\Omega$, one may always take
\begin{equation}\label{eq:sigma0-def}
\sigma_0
:=
\max\left\{
0,\,
-\min_{x\in\partial\Omega}
\min_{\tau\in T_x\partial\Omega:|\tau|=1}
\II_x(\tau,\tau)
\right\}
\le
\|\II\|_{L^\infty(\partial\Omega)}
<\infty .
\end{equation}
Thus the content of Theorem~\ref{thm:main-nc} below is the explicit
dependence of the decay rate on the negative part of the boundary
curvature, quantified by $\sigma_0$.

As in the convex case, hypotheses \textup{(i)}--\textup{(ii)} of
Theorem~\ref{thm:main-nc} are automatically satisfied for some finite
$m>0$ and $K\ge0$. Indeed, $\Omega$ is bounded and connected with
Lipschitz boundary, hence the Neumann Poincar\'e inequality holds for
Lebesgue measure on $\Omega$; since $U$ is continuous on the compact set
$\overline\Omega$, the density $e^{-U}$ is bounded above and below, and
therefore the corresponding weighted Poincar\'e inequality also holds.
Moreover, $\nabla^2U$ is continuous on $\overline\Omega$, so
$\nabla^2U\succeq-K\,\mathrm{Id}$ for some finite $K\ge0$.

\subsection{Well-posedness on a non-convex domain}

The geometric construction of specularly reflected Langevin processes
does not inherently require convexity. In particular,
\cite{BossyJabir15} treats bounded confinement domains whose boundary
is a compact $C^3$ submanifold of $\R^d$, without imposing convexity,
for bounded measurable drifts (our drift
$-\nabla U(x)-\gamma v$ is bounded in $x$ and linear in $v$, so this
covers \eqref{eq:sde} after a standard localization in the velocity
variable), and shows that the successive collision times do not
accumulate and that the process never reaches the grazing
set $\Gamma_0$ of \eqref{eq:grazing}, see \cite[\S2]{BossyJabir15}.
Nevertheless, Theorem~\ref{thm:main-nc} retains
Assumption~\ref{ass:WP}, because that assumption contains more than
pathwise or weak well-posedness of the reflected process: it also
requires the strongly continuous contraction semigroup on $L^2(\mu)$,
the identification of its generator with \eqref{eq:bke}, and the core
property of $\mathscr C_0$.

\subsection{The two inputs that referred to convexity}
\label{subsec:nc-inputs}
We first record that the functional-analytic framework of
Section~\ref{sec:setting} remains available without convexity.

\begin{lemma}[The Neumann realization on a $C^{1,1}$ domain]
\label{lem:neumann-nc}
Let $\Omega\subset\R^d$ be a bounded connected domain with
$\partial\Omega\in C^{1,1}$ and let $U\in C^2(\overline\Omega)$.
Let $-\Lo$ be the nonnegative self-adjoint operator on $L^2(\mux)$
associated with the closed form $\mathcal E$ on $H^1(\Omega)$.
Then
\begin{equation}\label{eq:neumann-domain-nc}
D(\Lo)
=
\left\{
h\in H^2(\Omega):
\partial_nh=0
\text{ $\sigma$-a.e.\ on }\partial\Omega
\right\},
\qquad
\Lo h=\Delta_xh-\nabla_xU\cdot\nabla_xh,
\end{equation}
and there exists $C=C(\Omega,U)$ such that
\begin{equation}\label{eq:apriori-H2-nc}
\|h\|_{H^2(\Omega)}
\le
C\left(
\nrm[\mux]{\Lo h}+\nrm[\mux]{h}
\right),
\qquad h\in D(\Lo).
\end{equation}
\end{lemma}
\begin{proof}
The proof is the same as that of Lemma~\ref{lem:neumann}, except for the argument of
elliptic regularity. In fact, convexity was used to obtain $H^2$
regularity of the variational solution of the Neumann problem from
$L^2$ data. On a bounded domain with $C^{1,1}$ boundary the
variational solution $h\in H^1(\Omega)$ of the Neumann problem
$\Delta h=F\in L^2(\Omega)$, $\partial_nh=0$, belongs to $H^2(\Omega)$
and satisfies
\[
\|h\|_{H^2(\Omega)}
\le
C\left(
\|\Delta h\|_{L^2(\Omega)}+\|h\|_{L^2(\Omega)}
\right),
\]
On the other hand, if the domain is not necessary convex,  this is the $L^p$ theory of the
Neumann (more generally, oblique) problem on $C^{1,1}$ domains in
\cite[\S2.3--\S2.4]{Grisvard85}, specialized to $p=2$; see
\cite[Thm.~2.3.3.2]{Grisvard85} for the regularity statement
$u\in W^{1,p}$, $\Delta u\in L^p$, $\partial_nu=0$ implies $u\in W^{2,p}$
on $C^{1,1}$ domains and \cite[Thm.~2.4.2.7]{Grisvard85} for the
corresponding existence and uniqueness statement in $W^{2,p}$, from which
the estimate follows by the closed graph theorem. In the scale of
\cite[Thm.~2.5.1.1]{Grisvard85}, already invoked with $k=1$ in Step~5 of
the proof of Lemma~\ref{lem:reilly}, this is the case $k=0$: a
$C^{k+1,1}$ boundary gives $H^{k+2}$ regularity for the Neumann problem
with $H^k$ data. Convexity is what allows one to lower the boundary
regularity from $C^{1,1}$ to Lipschitz in \cite[Thm.~3.2.1.3]{Grisvard85}.
Applying this to
\[
\Delta h=\Lo h+\nabla U\cdot\nabla h
\]
and absorbing the first-order term exactly as in the proof of
Lemma~\ref{lem:neumann} yields \eqref{eq:apriori-H2-nc}.
The reverse inclusion follows from the weighted Green formula
(Lemma~\ref{lem:green}) and also does not use convexity.
\end{proof}

\begin{remark}[The Reilly identity needs no convexity]
\label{rem:reilly-nc}
With Lemma~\ref{lem:neumann-nc} in hand, the weighted Reilly formula
\eqref{eq:reilly} of Lemma~\ref{lem:reilly} holds verbatim on a bounded
connected domain with $\partial\Omega\in C^{2,1}$.
Indeed, Steps~1--5 of its proof use only the pointwise
Bakry--\'Emery identity, the divergence theorem, the weighted Green
formula, tangential differentiation of
$\langle\nabla_xh,n\rangle=0$ on $\partial\Omega$, and elliptic
regularity for the Neumann problem (the $H^2$ a priori
estimate \eqref{eq:apriori-H2-nc} and the $H^3$ regularity of
\cite[Thm.~2.5.1.1]{Grisvard85} with $k=1$, both valid on $C^{2,1}$
domains without convexity). None of these requires convexity.
Convexity was used only in Step~6 to conclude that the boundary term has
a favorable sign.
Likewise, Lemmas~\ref{lem:antisym}, \ref{lem:antisym-sob},
\ref{lem:green}, Remark~\ref{rem:neumann-specular},
Propositions~\ref{prop:generator} and \ref{prop:invariance}, and the
bounds \eqref{eq:Am-bound}--\eqref{eq:LaAm-bound} of
Lemma~\ref{lem:corrector} do not use convexity. In the present section,
connectedness is assumed explicitly and ensures that
$\ker(-\Lo)=\operatorname{span}\{1\}$.
\end{remark}

\subsection{A boundary trace estimate with absorbable constant}
\label{subsec:nc-trace}
The boundary term left over from the weighted Reilly identity is
controlled by the following two lemmas. The first constructs a Lipschitz
extension of the outward normal, while the second turns the boundary
integral into interior norms with an arbitrarily small coefficient in
front of the second derivatives.

\begin{lemma}[Normal extension]\label{lem:normalfield-nc}
Let $\Omega\subset\R^d$ be bounded with $\partial\Omega\in C^{1,1}$ and
write
\[
r:=\reach(\partial\Omega)>0,
\]
for the reach of $\partial\Omega$ in the sense of Federer
\cite{Federer59}, i.e.\ the largest $r$ such that every point at distance
less than $r$ from $\partial\Omega$ has a unique nearest point on
$\partial\Omega$; a compact $C^{1,1}$ hypersurface has positive reach,
and $r$ is the largest radius for which $\Omega$ satisfies a uniform
interior \emph{and} exterior ball condition
\cite[Ch.~7]{DelfourZolesio11}.
For every $\delta\in(0,r/2]$, there exists
$\mathcal N\in W^{1,\infty}(\Omega;\R^d)$ such that
\begin{equation}\label{eq:N-props-nc}
|\mathcal N|\le1
\quad\text{on }\Omega,
\qquad
\mathcal N=n
\quad\text{$\sigma$-a.e.\ on }\partial\Omega,
\qquad
\supp\mathcal N
\subset
\left\{
x\in\Omega:
\dist(x,\partial\Omega)<\delta
\right\}.
\end{equation}
For such a vector field, define
\begin{equation}\label{eq:Theta-nc}
\Theta_{\mathcal N}
:=
\left\|
\nabla_x\cdot\mathcal N-\nabla_xU\cdot\mathcal N
\right\|_{L^\infty(\Omega)}
<\infty .
\end{equation}
\end{lemma}

\begin{proof}
Let
\[
\rho(x):=\dist(x,\partial\Omega).
\]
Since $\partial\Omega\in C^{1,1}$, it has positive reach and $\rho$ is
$C^{1,1}$ in a tubular neighborhood of $\partial\Omega$
(\cite[Thm.~4.8]{Federer59}, \cite[Ch.~7, Thm.~8.2]{DelfourZolesio11};
when $\partial\Omega\in C^k$ with $k\ge2$, which is the case under the
standing assumption $\partial\Omega\in C^{2,1}$ of this section, $\rho$ is
even $C^k$ on $\{x\in\overline\Omega:\rho(x)<r\}$ by
\cite[Lemma~14.16]{GilbargTrudinger01}). On the interior
collar one has $|\nabla\rho|=1$ and
\[
\nabla\rho=-n
\qquad\text{on }\partial\Omega,
\]
because $\rho$ increases in the inward normal direction.
Choose $\chi\in C^{0,1}([0,\infty))$ satisfying
\[
0\le\chi\le1,\qquad
\chi(0)=1,\qquad
\chi\equiv0\ \text{on }[\delta/2,\infty),
\qquad
|\chi'|\le\frac{2}{\delta},
\]
and define
\[
\mathcal N:=-\chi(\rho)\nabla\rho
\]
in the tubular neighborhood, extended by $0$ to the rest of $\Omega$.
Then $|\mathcal N|\le1$, $\mathcal N=n$ on $\partial\Omega$, and
$\supp\mathcal N\subset\{\rho<\delta\}$.
Furthermore,
\[
\nabla_x\cdot\mathcal N
=
-\chi'(\rho)|\nabla\rho|^2
-\chi(\rho)\Delta\rho,
\]
which belongs to $L^\infty(\Omega)$ because $\rho$ is $C^{1,1}$ on the
support of $\mathcal N$. Since $\nabla U$ is bounded on
$\overline\Omega$, \eqref{eq:Theta-nc} follows.
\end{proof}

\begin{remark}[Size of $\Theta_{\mathcal N}$]
\label{rem:Theta-size-nc}
Take $\delta=r/2$. On the support of $\mathcal N$ one has
$\rho\le r/4$. Let $\kappa_1,\dots,\kappa_{d-1}$ denote the
principal curvatures of $\partial\Omega$ with respect to the outward
normal (so that $\kappa_i\ge0$ on a convex domain), which exist
$\sigma$-a.e.\ and satisfy $|\kappa_i|\le1/r$ by the two-sided ball
condition of radius $r$. At a point $x$ of the collar with nearest
boundary point $y$, the Hessian of $\rho$ is, in principal coordinates at
$y$,
\[
\nabla^2\rho(x)
=
\operatorname{diag}\left(
-\frac{\kappa_1(y)}{1-\rho(x)\kappa_1(y)},\dots,
-\frac{\kappa_{d-1}(y)}{1-\rho(x)\kappa_{d-1}(y)},\,0\right),
\]
see \cite[Lemma~14.17]{GilbargTrudinger01} (the sign is dictated by
$\nabla\rho=-n$; for the ball $B_R$ one has $\rho=R-|x|$ and
$\Delta\rho=-(d-1)/(R-\rho)$). Since $\rho\le r/4$ and $|\kappa_i|\le1/r$
on the support of $\mathcal N$,
\[
|\Delta\rho|
\le
\sum_{i=1}^{d-1}\frac{|\kappa_i|}{1-\rho|\kappa_i|}
\le
\frac{(d-1)/r}{1-1/4}
=
\frac{4(d-1)}{3r}
\le
\frac{2(d-1)}{r}
\qquad\text{on }\supp\mathcal N,
\]
while $|\chi'|\le4/r$. Consequently,
\[
|\nabla_x\cdot\mathcal N|\le\frac4r+\frac{2(d-1)}{r}=\frac{2(d+1)}{r},
\]
and one may choose $\mathcal N$ such that
\begin{equation}\label{eq:Theta-explicit-nc}
\Theta_{\mathcal N}
\le
\frac{2(d+1)}{\reach(\partial\Omega)}
+
\|\nabla U\|_{L^\infty(\Omega)}.
\end{equation}
Notice that Lemma~\ref{lem:trace-nc}
below uses only the first two properties in \eqref{eq:N-props-nc}, so
that $\Theta_{\mathcal N}$ may be replaced throughout by the infimum of
\eqref{eq:Theta-nc} over all $\mathcal N\in W^{1,\infty}(\Omega;\R^d)$
with $|\mathcal N|\le1$ and $\mathcal N=n$ on $\partial\Omega$; the
support condition merely serves to produce the explicit bound
\eqref{eq:Theta-explicit-nc}.
\end{remark}

\begin{lemma}[Boundary trace absorption]\label{lem:trace-nc}
Let $\Omega\subset\R^d$ be bounded with $\partial\Omega\in C^{1,1}$,
let $U\in C^1(\overline\Omega)$, and let $\mathcal N$ and
$\Theta_{\mathcal N}$ be as in Lemma~\ref{lem:normalfield-nc}.
Then for every $h\in H^2(\Omega)$ and every $\varepsilon>0$,
\begin{equation}\label{eq:trace-nc}
\frac1{Z_x}
\int_{\partial\Omega}
|\nabla_xh|^2e^{-U}\,\rmd\sigma
\le
\varepsilon\nrm[\mux]{\nabla_x^2h}^2
+
\left(
\varepsilon^{-1}+\Theta_{\mathcal N}
\right)
\nrm[\mux]{\nabla_xh}^2.
\end{equation}
\end{lemma}
\begin{proof}
Assume first that $h\in C^2(\overline\Omega)$. Applying the divergence
theorem to
$|\nabla_xh|^2e^{-U}\mathcal N$,
and using $\mathcal N=n$ on $\partial\Omega$, we obtain
\[
\int_{\partial\Omega}|\nabla_xh|^2e^{-U}\,\rmd\sigma = \int_{\partial\Omega}|\nabla_x h|^2 e^{-U}\underbrace{(n\cdot \mathcal{N})}_{=|n|^2=1}\,\rmd\sigma
=
\int_\Omega
\nabla_x\cdot
\left(
|\nabla_xh|^2e^{-U}\mathcal N
\right)
\,\rmd x.
\]
Since
\[
\nabla_x|\nabla_xh|^2
=
2\nabla_x^2h\,\nabla_xh,
\]
we have
\[
\nabla_x\cdot
\left(
|\nabla_xh|^2e^{-U}\mathcal N
\right)
=
e^{-U}
\left[
2\nabla_x^2h(\nabla_xh,\mathcal N)
+
|\nabla_xh|^2
\left(
\nabla_x\cdot\mathcal N-\nabla_xU\cdot\mathcal N
\right)
\right].
\]
Dividing by $Z_x$ gives
\[
\frac1{Z_x}
\int_{\partial\Omega}
|\nabla_xh|^2e^{-U}\,\rmd\sigma
=
2
\int_\Omega
\nabla_x^2h(\nabla_xh,\mathcal N)\,\rmd\mux
+
\int_\Omega
|\nabla_xh|^2
\left(
\nabla_x\cdot\mathcal N-\nabla_xU\cdot\mathcal N
\right)
\,\rmd\mux.
\]
Since $|\mathcal N|\le1$,
\[
\left|
\nabla_x^2h(\nabla_xh,\mathcal N)
\right|
\le
\|\nabla_x^2h\|_F|\nabla_xh|.
\]
Hence the Cauchy--Schwarz inequality and
$2ab\le\varepsilon a^2+\varepsilon^{-1}b^2$ imply
\[
2
\left|
\int_\Omega
\nabla_x^2h(\nabla_xh,\mathcal N)\,\rmd\mux
\right|
\le
\varepsilon\nrm[\mux]{\nabla_x^2h}^2
+
\varepsilon^{-1}\nrm[\mux]{\nabla_xh}^2.
\]
The remaining term is bounded by
$\Theta_{\mathcal N}\nrm[\mux]{\nabla_xh}^2$.
This proves \eqref{eq:trace-nc} for smooth $h$.
For general $h\in H^2(\Omega)$, approximate $h$ in $H^2(\Omega)$ by
smooth functions (possible on a bounded Lipschitz domain).
The interior terms converge directly, while the
boundary term converges by the trace theorem applied to $\nabla h$.
\end{proof}

\subsection{The second-order estimate}\label{subsec:nc-hessian}
For $\theta\in(0,1)$, define
\begin{equation}\label{eq:Ktheta-nc}
K_\theta
:=
K+\frac{\sigma_0^2}{\theta}
+\sigma_0\Theta_{\mathcal N}.
\end{equation}

Now we can utilize Lemma \ref{lem:trace-nc} to obtain a second-order estimate that extends \eqref{eq:hessian-bound} to domains with semi-convex boundary. 

\begin{proposition}[Second-order estimate for a semiconvex boundary]
\label{prop:hessian-nc}
Let $\Omega$ be a bounded connected domain with
$\partial\Omega\in C^{2,1}$ satisfying
Assumption~\ref{ass:semiconvex}. Let
$U\in C^2(\overline\Omega)$ satisfy
\[
\nabla^2U\succeq-K\,\mathrm{Id},
\]
for some $K\ge0$, and let $\mathcal N,\Theta_{\mathcal N}$ be as in
Lemma~\ref{lem:normalfield-nc}. Then, for every $h\in D(\Lo)$ and every
$\theta\in(0,1)$,
\begin{equation}\label{eq:hessian-nc}
\nrm[\mux]{\nabla_x^2h}^2
\le
\frac1{1-\theta}
\left(
\nrm[\mux]{\Lo h}^2
+
K_\theta\nrm[\mux]{\nabla_xh}^2
\right).
\end{equation}
When $\sigma_0=0$ one has, in addition, the sharper bound
\eqref{eq:hessian-bound}, which is the limit $\theta\downarrow0$ of
\eqref{eq:hessian-nc}.
\end{proposition}

\begin{proof}
By Remark~\ref{rem:reilly-nc}, the weighted Reilly identity
\eqref{eq:reilly} remains valid:
\[
\nrm[\mux]{\nabla_x^2h}^2
=
\nrm[\mux]{\Lo h}^2
-
\int_\Omega
\langle
\nabla^2U\nabla_xh,\nabla_xh
\rangle
\,\rmd\mux
-
\frac1{Z_x}
\int_{\partial\Omega}
\II(\nabla_xh,\nabla_xh)e^{-U}\,\rmd\sigma.
\]
Since $h\in D(\Lo)$, $\partial_nh=0$, so $\nabla_xh$ is tangent to
$\partial\Omega$ almost everywhere. Therefore
Assumption~\ref{ass:semiconvex} gives
\[
-\II(\nabla_xh,\nabla_xh)
\le
\sigma_0|\nabla_xh|^2.
\]
Together with $\nabla^2U\succeq-K\,\mathrm{Id}$,
\begin{equation}\label{eq:pretrace-nc}
\nrm[\mux]{\nabla_x^2h}^2
\le
\nrm[\mux]{\Lo h}^2
+
K\nrm[\mux]{\nabla_xh}^2
+
\frac{\sigma_0}{Z_x}
\int_{\partial\Omega}
|\nabla_xh|^2e^{-U}\,\rmd\sigma.
\end{equation}
If $\sigma_0=0$, the last term vanishes and
\eqref{eq:hessian-bound} follows immediately; since
$K_\theta=K$ in that case and $1/(1-\theta)\ge1$, \eqref{eq:hessian-nc}
follows a fortiori.
Now suppose $\sigma_0>0$. Apply
Lemma~\ref{lem:trace-nc} with
$\varepsilon=\theta/\sigma_0$.
Then \eqref{eq:pretrace-nc} gives
\[
\nrm[\mux]{\nabla_x^2h}^2
\le
\nrm[\mux]{\Lo h}^2
+
K\nrm[\mux]{\nabla_xh}^2
+
\theta\nrm[\mux]{\nabla_x^2h}^2
+
\sigma_0
\left(
\frac{\sigma_0}{\theta}
+
\Theta_{\mathcal N}
\right)
\nrm[\mux]{\nabla_xh}^2.
\]
Absorbing the term
$\theta\nrm[\mux]{\nabla_x^2h}^2$ into the left-hand side (it
is finite since $h\in H^2(\Omega)$) and dividing
by $1-\theta$ gives \eqref{eq:hessian-nc}.
\end{proof}

\subsection{The corrector bound and the decay rate}
\label{subsec:nc-main}
Define
\begin{equation}\label{eq:A-def-nc}
A(\theta)
:=
\frac{2}{1-\theta}
\left(
1+\frac{K_\theta}{4m}
\right),
\qquad
A
:=
\inf_{\theta\in(0,1)}A(\theta).
\end{equation}
Then $A\ge2$. If $\sigma_0=0$, then
$A(\theta)=\frac{2}{1-\theta}\left(1+\frac{K}{4m}\right)$ is increasing in
$\theta$ and
\begin{equation}\label{eq:A-convex-limit}
A=\lim_{\theta\downarrow0}A(\theta)=2+\frac{K}{2m},
\end{equation}
which is exactly the constant appearing in
\eqref{eq:AmLa-bound}. If $\sigma_0>0$, then
$A(\theta)\to\infty$ both as $\theta\downarrow0$ and as
$\theta\uparrow1$, so the infimum is attained at an interior point
$\theta_*\in(0,1)$, computed explicitly in
Remark~\ref{rem:theta-nc}.

We now extend the bounds on the corrector in Lemma \ref{lem:corrector} to semiconvex domains in the following proposition. 

\begin{proposition}[Corrector bound for a semiconvex boundary]
\label{prop:corrector-nc}
Under the hypotheses of Proposition~\ref{prop:hessian-nc}, assume in
addition that $\mux$ satisfies the Poincar\'e inequality
\eqref{eq:poincare} with constant $m>0$. Then, for every
$\phi\in\mathscr C_0$,
\begin{equation}\label{eq:AmLa-nc}
\nrm{\Am\La(1-\Pv)\phi}
\le
\sqrt A\,\nrm{\phi}.
\end{equation}
The bounds \eqref{eq:Am-bound} and \eqref{eq:LaAm-bound} of
Lemma~\ref{lem:corrector} remain unchanged, as does the norm equivalence
\eqref{eq:equivalence}.
\end{proposition}
\begin{proof}
Steps~0--4 of the proof of Lemma~\ref{lem:corrector} are unchanged.
Indeed, they use only antisymmetry, the form identity
\eqref{eq:form-identity}, connectedness of $\Omega$, and Gaussian moment
computations.
Let
\[
B:=\Am\La(1-\Pv).
\]
For $g\in\Ran\Pv\cap\Lmuz$ with $\nrm g=1$, set
\[
h:=(m-\Lo)^{-1}g\in D(\Lo)\cap L^2_0(\mux).
\]
The computation \eqref{eq:Bstar-norm} still gives
\[
\nrm{B^*g}^2
=
2\nrm[\mux]{\nabla_x^2h}^2.
\]
For every $\theta\in(0,1)$, Proposition~\ref{prop:hessian-nc}
together with $\nrm[\mux]{\nabla_xh}^2=\ip[\mux]{h}{-\Lo h}$
yields
\[
\nrm{B^*g}^2
\le
\frac{2}{1-\theta}
\left(
\nrm[\mux]{\Lo h}^2
+
K_\theta
\ip[\mux]{h}{-\Lo h}
\right).
\]
Let $(E_\lambda)$ be the spectral resolution of $-\Lo$ on
$L^2_0(\mux)$. Since
\[
\Spec\left(-\Lo|_{L^2_0(\mux)}\right)\subset[m,\infty),
\]
and $h=(m-\Lo)^{-1}g$, we obtain
\[
\nrm{B^*g}^2
\le
\frac{2}{1-\theta}
\int_m^\infty
\frac{\lambda^2+K_\theta\lambda}{(m+\lambda)^2}
\,\rmd\nu_g(\lambda).
\]
Using
\[
\sup_{\lambda\ge m}
\frac{\lambda^2}{(m+\lambda)^2}
=1,
\qquad
\sup_{\lambda\ge m}
\frac{\lambda}{(m+\lambda)^2}
=
\frac1{4m},
\]
we get
\[
\nrm{B^*g}^2
\le
\frac{2}{1-\theta}
\left(
1+\frac{K_\theta}{4m}
\right)
\nrm[\mux]{g}^2
=
A(\theta)\nrm[\mux]{g}^2.
\]
Since this holds for every $\theta\in(0,1)$,
\[
\nrm{B^*g}^2\le A\nrm[\mux]{g}^2.
\]
The duality argument of Step~2 in the proof of
Lemma~\ref{lem:corrector} therefore gives \eqref{eq:AmLa-nc}.
The remaining corrector bounds do not use the second-order estimate and
remain unchanged.
\end{proof}

We have all the ingredients to show the explicit decay rate for the underdamped Langevin dynamics with specular reflection on a non-convex domain in the following theorem. 

\begin{theorem}[Decay rate on a non-convex domain]
\label{thm:main-nc}
Let $\Omega\subset\R^d$ be a bounded connected domain with $C^{2,1}$
boundary satisfying Assumption~\ref{ass:semiconvex} with constant
$\sigma_0\ge0$. Let $U\in C^2(\overline\Omega)$ and let $\mu$ be the
Gibbs measure \eqref{eq:gibbs}. Assume:
\begin{enumerate}
\item[(i)] $\mux$ satisfies the Poincar\'e inequality
\eqref{eq:poincare} with constant $m>0$;
\item[(ii)] $\nabla^2U\succeq-K\,\mathrm{Id}$ on $\Omega$ for some
$K\ge0$;
\item[(iii)] Assumption~\ref{ass:WP} holds.
\end{enumerate}
Fix a vector field $\mathcal N$ as in
Lemma~\ref{lem:normalfield-nc} and let $A$ be defined by
\eqref{eq:A-def-nc}. With friction coefficient
\begin{equation}\label{eq:gamma-nc}
\gamma
=
2\sqrt m\,\sqrt{A+2},
\end{equation}
the solution
$f(t)=e^{t\Lgen}f_0$
of \eqref{eq:bke} with the specular boundary condition and
$f_0\in\Lmuz$ satisfies, for every $t\ge0$,
\begin{equation}\label{eq:decay-nc}
\nrm{f(t)}
\le
\sqrt3\,e^{-\Lambda_\sigma t}\nrm{f_0},
\qquad
\Lambda_\sigma
=
\frac16
\frac{\sqrt m}{\sqrt A+\sqrt{A+2}}.
\end{equation}
\end{theorem}

\begin{proof}
Steps~1--5 in the proof of Theorem~\ref{thm:main} remain unchanged,
except that the corrector estimate \eqref{eq:AmLa-bound} is replaced by
\eqref{eq:AmLa-nc}. Thus the constant
$\sqrt{2+K/(2m)}$ in \eqref{eq:err1} is replaced by $\sqrt A$.
Consequently,
\[
\zeta
=
\frac{\gamma}{2\sqrt m}
+
\sqrt A,
\]
and
\[
\mathscr D_\varepsilon(f)
\ge
\begin{pmatrix}
\alpha&\beta
\end{pmatrix}
M_{\gamma,\varepsilon}
\begin{pmatrix}
\alpha\\
\beta
\end{pmatrix},
\qquad
M_{\gamma,\varepsilon}
=
\begin{pmatrix}
\gamma-\varepsilon
&
-\dfrac{\varepsilon\zeta}{2}
\\[3pt]
-\dfrac{\varepsilon\zeta}{2}
&
\dfrac{\varepsilon}{2}
\end{pmatrix},
\]
where
\[
\alpha:=\nrm{\nabla_vf},
\qquad
\beta:=\nrm{f_S}.
\]
Set
\[
s:=\sqrt A,
\qquad
u:=\sqrt{A+2},
\qquad
p:=s+u.
\]
Then $u^2=s^2+2$. Choose
$\gamma_*:=2\sqrt m\,u$,
and as in Section~\ref{sec:proof},
\[
\varepsilon_*
:=
\frac{\gamma_*}{a(\gamma_*)},
\qquad
a(\gamma):=2+\zeta^2.
\]
At $\gamma=\gamma_*$,
$\zeta=u+s=p$.
Moreover,
\[
2up
=
2us+2u^2
=
2us+2s^2+4
=
p^2+2.
\]
Therefore,
\[
a(\gamma_*)
=
p^2+2
=
2up,
\]
and hence,
\begin{equation}\label{eq:eps-nc}
\varepsilon_*
=
\frac{2\sqrt m\,u}{2up}
=
\frac{\sqrt m}{p}
=
\frac{\sqrt m}{\sqrt A+\sqrt{A+2}}.
\end{equation}
The entries of $M_{\gamma_*,\varepsilon_*}$ are
\[
M_{11}
=
\frac{\sqrt m}{p}(2up-1),
\qquad
M_{22}
=
\frac{\sqrt m}{2p},
\qquad
M_{12}
=
-\frac{\sqrt m}{2}.
\]
Since $A\ge2$, we have $u\ge2$ and $p\ge2+\sqrt2$. In particular,
\[
M_{11}-t
=
\frac{\sqrt m}{p}
\left(
2up-\frac54
\right)
>0,
\qquad
M_{22}-t
=
\frac{\sqrt m}{4p},
\]
where $t:=\frac{\sqrt m}{4p}$.
Using $2up-p^2=2$,
\begin{align*}
\det
\left(
M_{\gamma_*,\varepsilon_*}-t\,\mathrm{Id}
\right)
&=
\frac{m}{4p^2}
\left(
2up-\frac54-p^2
\right)
=
\frac{3m}{16p^2}
>0.
\end{align*}
Since the symmetric matrix $M_{\gamma_*,\varepsilon_*}-t\,\mathrm{Id}$ has
positive diagonal entries and positive determinant, it is positive
definite, i.e.
\begin{equation}\label{eq:lambda-coer-nc}
\lambda_{\min}
\left(
M_{\gamma_*,\varepsilon_*}
\right)
\ge
\frac{\sqrt m}
{4\left(\sqrt A+\sqrt{A+2}\right)}
=:
\lambda_{\mathrm{coer}}.
\end{equation}
Furthermore, $p>2$ implies
\[
\varepsilon_*<\frac{\sqrt m}{2}<\sqrt m,
\]
so \eqref{eq:equivalence} yields
\[
\frac14\nrm f^2
\le
\mathsf L_m(f)
\le
\frac34\nrm f^2.
\]
The remainder of Step~6 in the proof of
Theorem~\ref{thm:main} is unchanged. In particular,
\[
\mathscr D_{\varepsilon_*}(f)
\ge
\frac{4\lambda_{\mathrm{coer}}}{3}
\mathsf L_m(f),
\]
and therefore
\[
\frac{\rmd}{\rmd t}\mathsf L_m(f(t))
\le
-2\Lambda_\sigma
\mathsf L_m(f(t)),
\qquad
\Lambda_\sigma
=
\frac23\lambda_{\mathrm{coer}}.
\]
Using \eqref{eq:lambda-coer-nc},
\[
\Lambda_\sigma
=
\frac16
\frac{\sqrt m}{\sqrt A+\sqrt{A+2}}.
\]
Gr\"onwall's lemma and the norm equivalence then give
\[
\nrm{f(t)}
\le
\sqrt3\,e^{-\Lambda_\sigma t}\nrm{f_0}.
\]
The extension from the core to arbitrary $f_0\in\Lmuz$ follows by
density and strong continuity exactly as in
Theorem~\ref{thm:main}.
\end{proof}

\begin{remark}
When $\sigma_0=0$ in Theorem~\ref{thm:main-nc}, one has
$A=2+\frac{K}{2m}$,
and consequently
\[
\gamma
=
\sqrt{16m+2K},
\qquad
\Lambda_\sigma
=
\frac16
\frac{\sqrt m}
{\sqrt{2+\frac{K}{2m}}+\sqrt{4+\frac{K}{2m}}},
\]
so Theorem~\ref{thm:main-nc} reduces exactly to
Theorem~\ref{thm:main}.
\end{remark}

\begin{remark}[The matrix estimate is self-contained]
\label{rem:step6-selfcontained}
The preceding computation directly verifies
\[
\lambda_{\min}
\left(
M_{\gamma_*,\varepsilon_*}
\right)
\ge
\frac{\sqrt m}
{4\left(\sqrt A+\sqrt{A+2}\right)}
\]
for every $A\ge2$. In the special case
$A=2+K/(2m)$ it reproduces the coercivity estimate
\eqref{eq:lam-coer} used in the proof of
Theorem~\ref{thm:main}. Thus the final $2\times2$ matrix estimate does
not require a separate appeal to the optimization calculation of
\cite{FLL26}.
\end{remark}

\begin{remark}[Convergence of the law]
\label{rem:divergences-nc}
The conclusions of Corollary~\ref{cor:divergences} remain valid on the
non-convex domain, with $\Lambda$ replaced everywhere by
$\Lambda_\sigma$.
\end{remark}

\begin{remark}[Choosing $\theta$]\label{rem:theta-nc}
Write
$P
:=
1+
\frac{K+\sigma_0\Theta_{\mathcal N}}{4m}$
and
$Q
:=
\frac{\sigma_0^2}{4m}$.
Then $A(\theta)$ (defined in \eqref{eq:A-def-nc}) takes the form:
\[
A(\theta)
=
\frac{2}{1-\theta}
\left(
P+\frac Q\theta
\right).
\]
When $\sigma_0>0$, the unique minimizer in $(0,1)$ is given by
\begin{equation}\label{eq:theta-star-nc}
\theta_*
=
\frac{-Q+\sqrt{Q^2+PQ}}{P}.
\end{equation}
For an explicit bound that avoids optimization, one may take
$\theta=1/2$. This gives
\begin{equation}\label{eq:A-half-nc}
A
\le
4+
\frac{
K+2\sigma_0^2+\sigma_0\Theta_{\mathcal N}
}{m},
\end{equation}
and therefore,
\begin{equation}\label{eq:lambda-simple-nc}
\Lambda_\sigma
\ge
\frac{\sqrt m}{12\sqrt{A+2}}
\ge
\frac1{12}
\frac{m}
{\sqrt{
6m+K+2\sigma_0^2+\sigma_0\Theta_{\mathcal N}
}}.
\end{equation}
\end{remark}

\section{Conclusion}\label{sec:conclusion}

In this paper, we studied the underdamped Langevin dynamics confined to a bounded domain $\Omega\subset\R^d$ by specular reflection at the boundary, with invariant measure $\mu(\rmd x\,\rmd v)\propto e^{-U(x)-|v|^2/2}\,\rmd x\,\rmd v$ on $\Omega\times\R^d$. Assuming that the position marginal $\mux\propto e^{-U(x)}\rmd x$ satisfies a Poincar\'e inequality on $\Omega$ with constant $m>0$,  $\nabla^2U\succeq -K\,\mathrm{Id}$ on $\Omega$ for some $K\ge0$ and the domain $\Omega$ is convex, we extended the modified $L^2$ hypocoercivity method of Fan--Li--Lu \cite{FLL26}, based on the gap-shifted corrector \eqref{eq:gap-corrector}
to this boundary-value setting. Three structural observations allowed us to obtain our results in the constrained domain with a specular reflection:
\begin{itemize}
\item[(i)] the specular symmetry class renders the transport operator $\La$ antisymmetric;
\item[(ii)] the overdamped infinitesimal generator $\Lo$ entering the corrector is canonically realized with \emph{Neumann} boundary conditions, and this Neumann condition is precisely what makes all boundary terms in the corrector estimates vanish;
\item[(iii)] the Bochner identity used in the whole-space argument is replaced by a weighted Reilly formula, whose boundary contribution involves the second fundamental form of $\partial\Omega$ and is nonnegative for convex $\Omega$.
\end{itemize}
As a consequence, with friction $\gamma=\sqrt{16m+2K}$ we obtained the explicit decay rate 
$
\Lambda
$ (Theorem~\ref{thm:main}), which is $O(\sqrt m)$ when $U$ is convex; this scaling is optimal in the constrained setting, including the pure reflection case $U\equiv 0$ (Proposition~\ref{prop:sharp}). This yields exponential convergence of the law in $\chi^2$-divergence, total variation, and Wasserstein distance for the position marginal (Corollary~\ref{cor:divergences}).
The explicit decay rate $\Lambda$ for the convex domain $\Omega$ is the same as that in the unconstrained setting of \cite{FLL26}. Finally, we extended our results to the setting where the domain $\Omega$ is non-convex. We obtained an explicit contraction rate that depends on the domain (Theorem~\ref{thm:main-nc}).

\section*{Acknowledgments}
The authors would like to thank Andreas Eberle and Francis L\"{o}rler for helpful comments.
Qi Feng is partially supported by the grant DMS-2420029.
Lingjiong Zhu is partially supported by the grants NSF DMS-2053454 and DMS-2208303.

\bibliographystyle{plain}
\bibliography{langevin}

\bigskip
\noindent\textsc{Department of Mathematics and Computer Science, Fisk University, Nashville, TN}\\
\emph{Email address}: \texttt{hdu@fisk.edu}
\medskip
\noindent\textsc{Department of Mathematics, Florida State University, Tallahassee, FL}\\
\emph{Email address}: \texttt{qfeng2@fsu.edu}
\medskip
\noindent\textsc{Department of Mathematics, Florida State University, Tallahassee, FL}\\
\emph{Email address}: \texttt{lzhu2@fsu.edu}

\end{document}